\documentclass[10pt]{amsart}

\AtBeginDocument{%
   \def\MR#1{}
}

\usepackage[english]{babel}
\usepackage[T1]{fontenc}
\usepackage[utf8]{inputenc}

\usepackage[a4paper,top=3cm,bottom=3cm,left=2.7cm,right=2.7cm,marginparwidth=2cm]{geometry}

\usepackage{tikz}
\usepackage{tikz-cd}

\usepackage{amssymb}
\usepackage{amsmath}
\usepackage{amsfonts}
\usepackage{mathtools}

\usepackage{float}
\floatstyle{plaintop}
\restylefloat{table}

\usepackage[colorlinks=true, linkcolor=blue, citecolor=blue]{hyperref}
\usepackage{cleveref}

\theoremstyle{plain}
\newtheorem{theorem}{Theorem}[section]
\newtheorem{lemma}[theorem]{Lemma}
\newtheorem{corollary}[theorem]{Corollary}

\newtheorem{proposition}[theorem]{Proposition}

\theoremstyle{definition}
\newtheorem{definition}[theorem]{Definition}

\newtheorem{remark}[theorem]{Remark}
\newtheorem{example}[theorem]{Example}
\numberwithin{equation}{section}

\DeclareMathOperator{\Hom}{Hom}

\DeclareMathOperator{\coker}{coker}
\DeclareMathOperator{\rank}{\underline{rank}}
\DeclareMathOperator{\rep}{rep}
\DeclareMathOperator{\replf}{rep_{l.f.}}

\DeclareMathOperator{\gr}{\operatorname{Gr}}
\DeclareMathOperator{\dreplf}{decrep_{l.f.}}
\DeclareMathOperator{\drepfr}{decrep_{l.f.}^\circ}

\newcommand{\inj}{\mathcal{I}}

\title{Generic bases of skew-symmetrizable affine type cluster algebras}

\author{Lang Mou}
\author{Xiuping Su}

\address{Lang Mou\newline
Department of Mathematics, University of California Davis, \newline
One Shields Ave., Davis, CA 95616, USA}
\email{lmou.math@gmail.com}

\address{Xiuping Su\newline
Mathematical Sciences, University of Bath, \newline
Bath BA2 7AY, U.K.}
\email{xs214@bath.ac.uk}

\date{}

\begin{document}

\begin{abstract}
Geiss, Leclerc and Schr\"oer introduced a class of 1-Iwanaga--Gorenstein algebras $H$ associated to symmetrizable Cartan matrices with acyclic orientations, generalizing the path algebras of acyclic quivers. They also proved that indecomposable rigid $H$-modules of finite projective dimension are in bijection with non-initial cluster variables of the corresponding Fomin--Zelevinsky cluster algebra. In this article, we prove in all affine types that their conjectural Caldero--Chapoton type formula on these modules coincide with the Laurent expression of cluster variables. By taking generic Caldero--Chapoton functions on varieties of modules of finite projective dimension, we obtain bases for affine type cluster algebras with full-rank coefficients containing all cluster monomials.
\end{abstract}

\maketitle

\tableofcontents

\section{Introduction}\label{section: introduction}

For a triple $(C, D, \Omega)$ of a symmetrizable Cartan matrix $C$, a symmetrizer $D$ and an acyclic orientation $\Omega$ of $C$, Geiss, Leclerc and Schr\"oer \cite{MR3660306} introduced and initiated the study of a finite-dimensional algebra $H = H_K(C, D, \Omega)$ over a field $K$, generalizing the path algebra of an acyclic quiver. These algebras are 1-Iwanaga--Gorenstein \cite{MR0522552, MR0597688} and have subcategories $\replf H \subseteq \rep H$ consisting of \emph{locally free} modules (see \Cref{section: reps of h}). The modules in $\replf H$ can be characterized as finitely generated $H$-modules of finite projective dimension (thus at most 1). They play crucial roles in Hall algebra type realizations of the Lie algebra associated to $C$ and of the Fomin--Zelevinsky cluster algebra $\mathcal A(B)$ \cite{MR1887642} associated to the skew-symmetrizable matrix $B = B(C, \Omega) = (b_{ij})_{1\leq i, j\leq n}$, when $C$ is of finite type \cite{MR3555157, MR3830892}.

In this paper, we extend the connection between locally free modules $\replf H$ and the cluster algebra $\mathcal A(B)$ to the affine case, that is, when the Cartan matrix $C$ is of affine type. Affine type cluster algebras are fundamental in the classification and understanding of cluster algebras as they possess infinitely many distinct cluster variables while still exhibit periodic mutation behaviors, for instance, the mutation graph of exchange matrices is finite in this case \cite{felikson2012unfolding}.

Our first focus is proving a conjectural Caldero--Chapoton \cite{MR2250855} type formula of cluster variables proposed by Geiss--Leclerc--Schr\"oer in \cite{MR3830892}. It follows from \cite[Theorem 1.2(a)]{MR4125687} that (for any $C$ not restricted to affine types) sending a locally free $H$-module to its rank vector induces a bijection
\[
    \begin{tikzcd}
        \{\text{isoclasses of indecomposable rigid locally free $H$-modules}\}\ar[r, "\sim"] & \Delta_\mathrm{rS}(C, \Omega)
    \end{tikzcd}
\]
where the later denotes the set of real (positive) Schur roots with respect to $\Omega$ in the root system associated to $C$. In the cluster algebra $\mathcal A(B)$, non-initial cluster variables are parametrized by their $\mathbf d$-vectors, which are exactly the real Schur roots $\Delta_\mathrm{rS}(C, \Omega)$; see \Cref{subsec: real schur roots} and \Cref{thm: d vectors real schur roots}. By composing the two aforementioned bijections, we have a correspondence between indecomposable rigid locally free $H$-modules $M(\beta)$ and non-initial cluster variables $X_\beta$, both labeled by real Schur roots $\beta\in \Delta_{\mathrm{rS}}(C, \Omega)$. Our first main result is the following theorem expressing cluster variables directly from their corresponding modules. We take $K = \mathbb C$ and $\gr_\mathbf r^H(M)$ denotes the quasi-projective complex variety of subobjects of $M$ in $\replf H$ with rank vector $\mathbf r$. The notation $\chi(\cdot)$ takes the Euler characteristic in analytic topology.

\begin{theorem}[\Cref{thm: main bijection cc function}]\label{thm: first main theorem intro}
    Let $\beta = \sum_{i=1}^n m_i \alpha_i$ be a real Schur root with $\{\alpha_1, \dots, \alpha_n\}$ the positive simple roots and $M = M(\beta)$ be the unique indecomposable rigid locally free $H$-module whose rank vector is $(m_i)_{i=1}^n$. If $C$ is of affine type, then the Caldero--Chapoton function
    \begin{equation}\label{eq: cc function introduction}
        X_M \coloneqq \prod_{i=1}^n x_i^{-m_i} \cdot \sum_{\mathbf r\in \mathbb N^n} \chi(\gr_\mathbf r^H(M))\prod_{i=1}^n x_i^{\sum_{j=1}^n ([-b_{ij}]_+m_j + b_{ij}r_j)} \in \mathbb Z[x_1^\pm, \dots, x_n^\pm]
    \end{equation}
    equals the cluster variable $X_\beta \in \mathcal A(B)$ whose $\mathbf d$-vector is $(m_i)_{i=1}^n$.
\end{theorem}

The formula (\ref{eq: cc function introduction}) was proposed by Geiss, Leclerc and Schr\"oer in \cite{MR3830892}. It is a generalization of Caldero--Chapoton's original formula \cite{MR2250855} for Dynkin quivers to the skew-symmetrizable case in the context of $H$-modules. The formula was proven to be a cluster variable in \cite{MR3830892} when $C$ is of finite type. Since then it has been verified for all rank 2 cluster algebras in \cite{MR4732400} and for affine type $\widetilde C_n$ with the minimal symmetrizer in \cite{labardini2023gentle}.

As already pointed out in \cite{MR3830892}, it does not seem possible to extend the proof of \cite{MR2250855} (or \cite{MR2385670, MR2316979}) to our case. Instead, we seek a recurrence for the Laurent expansions of cluster variables by changing the base seed, a strategy taken in \cite{MR2629987}. Our proof of the formula (\ref{eq: cc function introduction}) uses the full power of such a recurrence under sink/source mutations developed in \cite{MR4732400} which turns out to be sufficient in all affine types.

Another key aspect of cluster algebra research is the study of their bases. Our next goal is to construct a basis $\mathcal S$ for the cluster algebra $\mathcal A(\widetilde B)$ with coefficients where $\widetilde B$ is an $m\times n$ exchange matrix extended from $B$. We impose the assumption that $\widetilde B$ is of full rank. Instead of taking the Caldero--Chapoton (CC) function of particular modules, we consider the affine space $\replf(H, \mathbf r)$ of $H$-modules of a fixed rank vector $\mathbf r$ and take the generic CC function on $\replf(H, \mathbf r)$ (see \Cref{subsec: generic f poly}). The elements $X^{\widetilde B}_{\tilde {\mathbf g}}$ in $\mathcal S$ will be parametrized by vectors $\tilde {\mathbf g}\in \mathbb Z^m$. For each $\tilde {\mathbf g}$ there is a corresponding rank vector $\mathbf v^+\in \mathbb N^n$. Then $X^{\widetilde B}_{\tilde {\mathbf g}}$ defined as in (\ref{eq: generic cc function coefficients}) takes the generic value on $\replf(H, \mathbf v^+)$ with a slightly modified form from (\ref{eq: cc function introduction}) to incorporate coefficients. Our second main result is

\begin{theorem}[\Cref{thm: generic basis}]\label{thm: second main theorem intro}
    If $C$ of affine type, the set
    \[
        \mathcal S = \{X^{\widetilde B}_{\tilde {\mathbf g}} \mid \tilde {\mathbf g} \in \mathbb Z^m \} \subseteq \mathbb Z[x_1^\pm, \dots, x_m^\pm]
    \]
    is a $\mathbb Z$-basis of the cluster algebra $\mathcal A(\widetilde B)$ containing all cluster monomials.
\end{theorem}

The approach of taking generic CC functions on varieties of quiver representations dates back to Dupont \cite{MR2738377} with a focus on affine quivers and especially in affine type $A$ where the generic CC functions were proven to form a basis of the cluster algebra. The basis property was extended by Ding, Xiao and Xu \cite{MR3036003} to any affine acyclic quivers. Geiss, Leclerc and Schr\"oer proved in \cite{MR2833478} that generic CC functions form a $\mathbb C$-basis of $\mathcal A(Q)\otimes_\mathbb Z \mathbb C$ when $Q$ is any acyclic quiver. Plamondon \cite{MR3061943} then constructed more generally for quivers with cycles generic CC functions parametrized by $\mathbf g$-vectors, a way of parametrization we shall take in \Cref{section: generic bases}. We note that for finite type cluster algebras, these generic functions are exactly cluster monomials. So our result can be regarded as a first attempt to generalize the construction of generic bases to the skew-symmetrizable case of infinite type using $H$-modules. The proof of \Cref{thm: second main theorem intro} relies on \Cref{prop: reflection generic f poly} the recurrence of generic functions under reflection functors, which works in general not restricted to affine types. The proof also utilizes a powerful theorem of Qin \cite{MR4721032} characterizing nice bases of \emph{injective-reachable} cluster algebras which include the affine case.

There have been various other efforts to extend representation-theoretic methods to study skew-symmetrizable cluster algebras, especially in the acyclic case such as \cite{hubery2006acyclic, MR2431998, MR2844757, MR2817677, MR3378824}. In the cause of obtaining coefficients in the Laurent expansion of cluster variables (and also generic functions), the usage of the varieties $\gr_\mathbf r^H(M)$ is novel in the current endeavor initiated in \cite{MR3830892}.

The paper is organized as follows. We first review in \Cref{section: reps of h} some representation-theoretic aspects of GLS algebras $H_K(C, D, \Omega)$. Notions that are important throughout the paper such as locally free modules and generalized reflection functors are explained in \Cref{section: reps of h}. Then we study in \Cref{section: coxeter action} the root systems of affine types and the Coxeter transformations on them. Necessary preliminaries of cluster algebras including $\mathbf g$-vectors and $F$-polynomials will be reviewed in \Cref{section: cluster and coxeter}. In \Cref{section: cc function}, we prove our first main result on the Caldero--Chapoton formula for cluster variables. The generic CC functions are constructed and proven to form bases of affine type cluster algebras in \Cref{section: generic bases}.

\section{\texorpdfstring{Representations of $H_K(C, D, \Omega)$}{Representations of H(C,D,Omega)}}\label{section: reps of h}

\subsection{\texorpdfstring{The algebra $H_K(C, D, \Omega)$}{The algebra H(C, D, Omega)}}\label{subsec: def of h}

We start by reviewing the construction of the Geiss--Leclerc--Schr\"oer (GLS) algebra $H = H_K(C, D, \Omega)$ associated to a Cartan matrix $C = (c_{ij})\in \mathbb Z^{n\times n}$, a left symmetrizer $D = \mathrm{diag}(c_i\mid i=1, \dots n)$ of $C$ and an orientation $\Omega$. Further details are referred to the original article \cite{MR3660306}.

Let $I = \{1, \dots, n\}$. The \emph{orientation} $\Omega$ of $C$ is a subset of $I\times I$ such that for any $i$ and $j$,
\[
    |\{(i, j), (j,i)\} \cap \Omega|  = \begin{dcases}
        1 & \text{if $c_{ij} < 0$} \\
        0 & \text{otherwise}.
    \end{dcases}
\]
Let $G(\Omega)$ be the oriented graph (or quiver) with vertex set $I$ and edges $(i, j)$ in $\Omega$ pointing from $j$ to $i$. Throughout the paper we only consider \emph{acyclic} orientations meaning that $G(\Omega)$ does not have oriented cycles.

Let $H_i = K[\varepsilon_i]/\varepsilon_i^{c_i}$ for each $i\in I$ with a fixed ground field $K$. For every $(i, j)\in \Omega$, define the $(H_i, H_j)$-bimodule
\[
    {_i H_j} \coloneqq \left( (H_i \otimes_K H_j)/(\varepsilon_i^{-c_{ji}/g_{ij}}\otimes 1 - 1 \otimes \varepsilon_j^{-c_{ij}/g_{ij}}) \right)^{\oplus^{g_{ij}}} \quad \text{with $g_{ij}\coloneqq \gcd(c_{ij}, c_{ji})$}.
\]
It is free of rank $-c_{ij}$ as left $H_i$-module and free of rank $-c_{ji}$ as right $H_j$-module. The algebra $H = H_K(C, D, \Omega)$ is defined to be the tensor algebra over the ring $S \coloneqq \prod_{i\in I} H_i$ of the $(S, S)$-bimodule $\bigoplus_{(i, j)\in \Omega} {_i H_j}$, that is,
\[
    H = \bigoplus_{k\geq 0} \left( \bigoplus_{(i, j)\in \Omega} {_i H_j} \right)^{\otimes_S^k}.
\]
Viewed as a $K$-algebra, $H$ is finite-dimensional. It has a description \cite[Section 1.4]{MR3660306} as the path algebra of a quiver $Q(C, \Omega)$ quotient by an ideal $I$.

\subsection{\texorpdfstring{Representations of $H_K(C, D, \Omega)$}{Representations of H(C,D,Omega)}}\label{subsec: reps of H}

Let $\rep H$ denote the category of finitely generated left $H$-modules. It is isomorphic to the category $\rep(C, D, \Omega)$ of representations of a \emph{modulated graph} $(H_i, {_i H_j})_{i\in I, (i, j)\in \Omega}$ \cite[Section 5.2]{MR3660306}. An object $M = (M_i, M_{ij}) \in \rep(C, D, \Omega)$ consists of a finitely generated $H_i$-module $M_i$ for each $i\in I$ and an $H_i$-morphism 
\[
    M_{ij}\colon {_i H_j}\otimes_{H_j} M_j \rightarrow M_i\quad \text{for each} \quad (i, j)\in \Omega.
\]
A morphism $f\colon M \rightarrow N$ is a tuple $f = (f_i \colon M_i \rightarrow N_i)_{i\in I}$ of $H_i$-morphisms intertwining with $M_{ij}$ and $N_{ij}$, that is, $f_i\circ M_{ij} = N_{ij} \circ(\mathrm{id}_{_iH_j} \otimes f_j)$ for any $(i, j)\in \Omega$.

For $(i, j)\in \Omega$, we define ${_j H_i}$ to be just ${_i H_j}$ but viewed as $H_j$-$H_i$-bimodule. There is an isomorphism from $\Hom_{H_i}({_iH_j}, H_i)$ to ${_jH_i}$ as $H_j$-$H_i$-bimodules which we shall fix (see \cite[Section 5.1]{MR3660306}). This isomorphism together with the tensor-hom adjunction induces an isomorphism
\begin{equation}\label{eq: iso between hom spaces}
    \Hom_{H_i}({_iH_j}\otimes_{H_j} M_j, M_i) = \Hom_{H_j}(M_j, {_jH_i\otimes_{H_i}M_i}).
\end{equation}
Therefore instead of using $M_{ij}$ in determining the module structure of $M$, one can equivalently use its counterpart $\overline M_{ij}\in \Hom_{H_j}(M_j, {_jH_i\otimes_{H_i}M_i})$ through the isomorphism (\ref{eq: iso between hom spaces}).

A more `compact' way to describe $M$ is to record for every $i\in I$ the $H_i$-morphism
\begin{equation}\label{eq: incoming mor}
    M_{i,\mathrm{in}} \coloneqq (M_{ij})_j \colon \bigoplus_{j\colon\!(i, j)\in \Omega} {_i H_j} \otimes_{H_j} M_j \rightarrow M_i,\quad \text{or}
\end{equation}
\begin{equation}\label{eq: outgoing mor}
    M_{i,\mathrm{out}} \coloneqq (\overline M_{ji})_j \colon M_i \rightarrow \bigoplus_{j\colon\!(j,i)\in \Omega} {_i H_j}\otimes_{H_j} M_j.
\end{equation}
These data will be useful in defining reflection functors in \Cref{subsec: reflection functor}.

\subsection{Locally free modules}\label{subsec: loc free module}

The notion of \emph{locally free module} is essential throughout the paper. An $H$-module $M\in \rep H$ is called \emph{locally free} if every $M_i$ is free (of finite rank) over $H_i$. In this case, the tuple $\rank M \coloneqq (\mathrm{rank}_{H_i}M_i)_i\in \mathbb N^n$ is called its \emph{rank vector}.

Locally free $H$-modules form a full subcategory $\replf H \subset \rep H$. They can be characterized as objects in $\rep H$ of finite injective dimension (equivalently at most $1$), equivalently of finite projective dimension (at most $1$), as $H$ is $1$-Iwanaga--Gorenstein; see \cite[Theorem 1.2]{MR3660306}. In particular, any projective or injective $H$-module in $\rep H$ is locally free.

\begin{definition}\label{def: inj and proj g-vector}
    For $M\in \rep H$ and a minimal injective co-presentation (thus a co-resolution for $M \in \replf H$ as in this case $\operatorname{inj. dim} M \leq 1$)
    \[
        0 \rightarrow M \rightarrow \bigoplus \inj_i^{a_i} \rightarrow \bigoplus \inj_i^{b_i},
    \]
    the tuple 
    \[
        \mathbf g^\mathrm{inj}_H(M) = \mathbf g^\mathrm{inj}(M) \coloneqq (b_i - a_i)_i \in \mathbb Z^n
    \]
    is called the \emph{injective $\mathbf g$-vector} of $M$.
\end{definition}

\begin{definition}\label{def: matrix b}
    We define the $n\times n$ skew-symmetrizable matrix $B = B(C, \Omega) = (b_{ij})$ by setting $b_{ij} = -c_{ij}$ and $b_{ji} = c_{ji}$ if $(i, j)\in \Omega$ and the rest entries $0$. 
\end{definition}

The following lemma follows directly from \cite[Prop 3.1 and Prop 3.5]{MR3660306}.

\begin{lemma}\label{lemma: g vector from rank vector}
    For any $M\in \replf H$ with $\rank M = (m_i)_i$, we have
    \[
        \mathbf g^\mathrm{inj}(M) = (-m_i + \sum_{j\in I} [-b_{ij}]_+ m_j)_i.
    \]
\end{lemma}

The rank vector of a locally free $H$-module $M$ can be recovered from $\mathbf g^\mathrm{inj}(M) = (g_i)_i$ in the following way. Since $\Omega$ is acyclic, using \Cref{lemma: g vector from rank vector}, one can start with $m_i = -g_i$ for a vertex $i$ without any successor (called a \emph{sink}) in the graph $G(\Omega)$ and inductively solving $m_k$ from equations
\begin{equation}\label{eq: from g vector to rank}
    m_k = -g_k + \sum_{j\in I} [-b_{kj}]_+ m_j.
\end{equation}

\subsection{Reflection functors and Coxeter functors}\label{subsec: reflection functor}

The classical Bernstein--Gel'fand--Ponomarev \cite{MR0393065} reflection functor is an efficient way to relate representations of different acyclic orientations of Dynkin diagrams that simulates simple reflections on the root lattice. We here review Geiss--Leclerc--Schr\"oer's generalized reflection functor \cite[Section 9]{MR3660306} on $H$-modules.

An index $i\in I$ is a \emph{sink} (resp. \emph{source}) of $\Omega$ if $(j, i)\notin \Omega$ (resp. $(i, j)\notin \Omega$) for any $j$, that is, a \emph{sink} (resp. \emph{source}) in the oriented graph $G(\Omega)$. The \emph{reflection} $s_i$ of $\Omega$ at $i$ is the orientation
\[
    s_i(\Omega) \coloneqq \Omega \setminus \{(i, j), (j, i) \mid j\in I\} \cup \{(j, i) \mid (i, j)\in \Omega\} \cup \{(i, j) \mid (j, i)\in \Omega\}.
\]
Namely, the graph $G(s_i(\Omega))$ is simply reversing edges in $G(\Omega)$ incident to $i$. We will only perform reflections at a sink or source, in which case $s_i(H)$ denotes $H(C, D, s_i(\Omega))$.

We define \emph{reflection functors} 
\[
    F_i^+ \ (\text{resp. } F_i^- )\  \colon \rep H \rightarrow \rep s_i(H)
\]
when $i$ is a sink (resp. source) as follows. Let $M$ be in $\rep H$ and $i$ be a sink. Consider the $H_i$-morphism $M_{i, \mathrm{in}}$ as in (\ref{eq: incoming mor}). The $s_i(H)$-module $M' = F_i^+(M)$ is defined to have $M'_j = M_j$ for $j\neq i$ and $M'_i = \ker M_{i, \mathrm{in}}$. Furthermore, the structure morphism  
\[
    M'_{i, \mathrm{out}} \colon M'_i \ \rightarrow \bigoplus_{j\colon\! (j, i)\in s_i(\Omega)} {_iH_j}\otimes_{H_j} M'_j
\]
is defined to be the natural inclusion 
\[
    \ker M_{i, \mathrm{in}} \hookrightarrow \bigoplus_{j\colon\! (i, j)\in \Omega} {_iH_j}\otimes_{H_j} M_j = \bigoplus_{j\colon\! (j,i)\in s_i(\Omega)} {_iH_j}\otimes_{H_j} M'_j;
\]
and $M'_{kj} \coloneqq M_{kj}$ if neither $k$ nor $j$ is $i$.

When $i$ is a source, then $M' = F_i^-(M)$ has $M'_j = M_j$ for $j\neq i$ and $M'_i = \coker M_{i, \mathrm{out}}$. The structure morphism $M'_{i, \mathrm{in}}$ is defined to be the natural projection
\[
    \bigoplus_{j\colon\! (j, i)\in \Omega} {_iH_j}\otimes_{H_j} M_j = \bigoplus_{j\colon\! (i,j)\in s_i(\Omega)} {_iH_j}\otimes_{H_j} M'_j \rightarrow \coker M_{i; \mathrm{out}};
\]
and $M'_{kj} \coloneqq M_{kj}$ if neither $k$ nor $j$ is $i$.

Let $E_i\in \replf H$ be $(E_i)_i = H_i$ and $(E_i)_j = 0$ for $j\neq i$. It will be referred to as the \emph{pseudo-simple} module at $i\in I$. It is then clear that any $E_i^{\oplus m}$ is annihilated by $F_i^\pm$ when the reflection functors apply.

\begin{proposition}[{\cite[Proposition 9.6]{MR3660306}}]\label{prop: reflection of loc free and rigid}
    Let $M\in \rep H$ be locally free and rigid. Then $F_k^\pm(M)$ is locally free and rigid.
\end{proposition}

Without loss of generality, we assume that $(i, j)\in \Omega$ implies that $i < j$. Therefore the vertex $1$ is always a sink and $n$ always a source. Notice that $s_n\cdots s_1(\Omega) = s_1\cdots s_n(\Omega) = \Omega$. Define the \emph{Coxeter functors}
\[
    C^+ \coloneqq F_n^+ \circ \cdots \circ F_1^+ \quad \text{and} \quad C^- \coloneqq F_1^- \circ \cdots \circ F_n^- \colon \rep H \rightarrow \rep H.
\]

There is also the \emph{twist automorphism functor} $T \colon \rep H \rightarrow \rep H$ such that
\[
    T(M)_i = M_i\quad \text{and} \quad T(M)_{ij} = -M_{ij} \coloneqq {_iH_j}\otimes_{H_j} M_j \rightarrow M_i,
\]
and $T$ does not change morphisms.

\begin{theorem}[{\cite[Theorem 10.1]{MR3660306}}]\label{thm: coxeter functors AR translation}
    For any $M\in \replf H$, there are functorial isomorphisms
    \[
        TC^+(M) \cong \tau(M) \quad \text{and} \quad TC^-(M) \cong \tau^-(M)  
    \]
    where $\tau$ (resp. $\tau^-$) is the Auslander--Reiten (resp. inverse) translation of $H$-modules.
\end{theorem}

\begin{definition}\label{def: tau loc free}
    An $H$-module $M$ is called \emph{$\tau$-locally free} if $\tau^k(M)$ is locally free for any $k\in \mathbb Z$.
\end{definition}

We set up a lattice $\mathbb Z^n$ with a symmetric bilinear form $(\alpha_i, \alpha_j) = c_ic_{ij}$ on the standard basis $\{\alpha_1, \dots, \alpha_n\}$. Associated to each $\beta \in \mathbb R^n$ with $(\beta, \beta)\neq 0$ is the reflection
\[
    s_{\beta} \in \operatorname{GL}_n(\mathbb R),\quad s_{\beta}(u) \coloneqq u - \frac{2(\beta, u)}{(\beta, \beta)} \beta \quad {\text{for $u\in \mathbb R^n$}}.
\]
The \emph{simple reflections} $s_i\coloneqq s_{\alpha_i}$ for $i\in I$ generate the Weyl group $W = W(C) < \operatorname{GL}_n(\mathbb R)$. The linear transformation $c = s_1s_2\cdots s_n$ is called the \emph{Coxeter element} (in $W$) associated to $\Omega$. Notice that the action of $W$ preserves the bilinear form $(-, -)$.

An \emph{admissible sequence} for $(C, \Omega)$ is a tuple $((i_1, p_1),\dots, (i_t, p_t)) \in (I \times \{+, -\})^t$ such that for any $1 \leq s \leq t$, either $i_s$ is a sink of $\Omega_s \coloneqq s_{i_{s-1}}\cdots s_{i_0}(\Omega)$ and $p_s = +$ or $i_s$ is a source of $\Omega_s$ and $p_s = -$, where $s_{i_0}(\Omega)$ formally denotes $\Omega$.

The following proposition summarizes properties of $\tau$-locally free modules under reflections.

\begin{proposition}\label{prop: tau loc free reflection}
    Let $M \in \rep H$ be indecomposable and $\tau$-locally free. 
    \begin{enumerate}
        \item For any admissible sequence $((i_1, p_1),\dots, (i_t, p_t))$ for $(C, \Omega)$, the module $N  = F_{i_t}^{p_t}\cdots F_{i_1}^{p_1}(M)$ is again indecomposable and $\tau$-locally free.
        \item Let $k$ be a sink (resp. source) of $\Omega$, then either $M = E_k$ or the $H_k$-morphism $M_{k,\mathrm{in}}$ (resp. $M_{k, \mathrm{out}}$) is surjective (resp. injective).
        \item If $N\neq 0$, then $\rank N = s_{i_t}\cdots s_{i_1}(\rank M)$. In particular, if $\tau M \neq 0$, then $\rank (\tau M) = c^{-1}(\rank M)$.
    \end{enumerate}
\end{proposition}

\begin{proof}
    The statement (1) was proven in \cite[Proposition 11.8]{MR3660306}. 
    
    For (2), let $k$ be a sink and assume that $M\neq E_k$. Then $F_k^+(M)$ is locally free by (1), which implies that the image of $M_{k, \mathrm{in}}$ in $M_k$ is a free $H_k$-submodule. Thus $M_{k, \mathrm{in}}$ is surjective; otherwise $M$ is decomposable. The case of $k$ being a source is dual. 
    
    For (3), it suffices to prove for a single reflection $s_i$. The calculation on rank vectors follows directly from the definition of reflection functors and (2). If $\tau M \neq 0$, then $\rank(\tau M) = \rank (C^+(M))$ since the twist $T$ does not change the rank vector and the latter equals $c^{-1}(\rank (M))$.
\end{proof}

We have a large class of modules that are $\tau$-locally free by

\begin{proposition}[{\cite[Proposition 11.4]{MR3660306}}]\label{prop: rigid implies tau loc free}
    Every locally free rigid $H$-module is $\tau$-locally free.
\end{proposition}

Therefore any preprojective or preinjective module is $\tau$-locally free by \Cref{prop: reflection of loc free and rigid}. In the rest of the paper, we will mostly use the $\tau$-locally free property of locally free rigid modules.

\section{Coxeter transformations on affine root systems}\label{section: coxeter action}

\subsection{Real Schur roots}\label{subsec: real schur roots}

The set of \emph{real roots} is
\[
    \Delta_\mathrm{re}(C) \coloneqq \{ w(\alpha_i) \mid w\in W,\ i\in I\}
\]
where $\{\alpha_1, \dots, \alpha_n\}$ are the \emph{simple roots}. A real root is called \emph{positive} if it is a non-negative linear combination of $\alpha_1, \dots, \alpha_n$. Denote by $\Delta_\mathrm{re}^+(C)$ the set of positive real roots.

The \emph{absolute length} $l(w)$ of $w\in W$ is the minimal $r\in \mathbb N$ such that $w$ can be expressed as $w = s_{\beta_1}s_{\beta_2}\dots s_{\beta_r}$ with every $\beta_i$ being a real root. We define the \emph{absolute order} $\leq $ on $W$ by
\[
    u \leq v \iff l(u) + l(u^{-1}v) = l(v).
\]
The set of \emph{real (positive) Schur roots} is defined as
\begin{equation}\label{eq: real schur roots}
    \Delta_\mathrm{rS}(C, \Omega) \coloneqq \{ \beta \in \Delta_\mathrm{re}^+(C) \mid s_\beta \leq c \}.
\end{equation}

The following theorem is an interpretation of real Schur roots in terms of $H$-modules.

\begin{theorem}[\cite{MR4125687}]\label{thm: real schur roots rigid module}
    Real Schur roots are in bijection with (isomorphism classes of) indecomposable locally free rigid $H_K(C, D, \Omega)$-modules as their rank vectors.
\end{theorem}

\subsection{Orbits of the Coxeter transformation}\label{subsection: coxeter orbit}

A Cartan matrix is said to be of \emph{affine type} if the symmetric form $(-, -)$ is positive semi-definite but not positive definite and of \emph{finite type} if positive definite. Finite type Cartan matrices are classified by Dynkin diagrams. There is a classification of affine Cartan matrices given by affine Dynkin diagrams in \cite{MR1104219} (see also \cite{MR0255627}).

We are interested in the action of the Coxeter element $c$ on $\Delta_\mathrm{re}$ when $C$ is of affine type.

\begin{proposition}[\cite{MR0447344}]\label{prop: infinite c orbits}
    There are exactly $2n$ infinite $c$-orbits in $\Delta_\mathrm{re}$ which are
    \[
        \{c^r \beta_i \mid r\in \mathbb Z\}\quad \text{and}\quad \{c^r \gamma_i \mid r\in \mathbb Z\} \quad \text{for $i\in I$},
    \]
    where
    \[
        \beta_i \coloneqq s_1s_2\cdots s_{i-1}(\alpha_i) \quad \text{and} \quad \gamma_{i} \coloneqq s_ns_{n-1}\cdots s_{i+1}(\alpha_i).
    \]
\end{proposition}

The positive roots in these infinite orbits are $c^r \beta_i$ for $r\geq 0$ and $c^r \gamma_i$ for $r\leq 0$. As we shall see next, these positive roots are realized as the rank vectors of preprojective and preinjective $H$-modules. This can be derived from the results in \cite{MR3660306}. For completeness we describe a construction in the proof below.

\begin{proposition}
    Let $P_i$ and $\inj_i$ be respectively the projective cover and injective envelope in $\rep H$ of the simple module $S_i$ at $i\in I$. Then we have for any $r\in \mathbb N$
    \[
        c^r \beta_i = \rank {(C^-)^r(P_i)} \quad \text{and} \quad c^{-r} \gamma_i = \rank {(C^+)^r(\inj_i)}.
    \]
\end{proposition}

\begin{proof}
    The projective module $P_i$ can be constructed as follows. Let $\Omega' = s_{i-1}\cdots s_2s_1(\Omega)$ and $H' = H_K(C, D, \Omega')$. Then $P_i = F_1^-\cdots F_{i-1}^-(E_i)$ where $E_i\in \rep H'$ is the rank one locally free module at $i$, i.e. $\rank E_i = \alpha_i$. Then it follows from \Cref{prop: tau loc free reflection} that $\beta_i = \rank P_i$. The statement for rest $c^r \beta_i$ then again follows from \Cref{prop: tau loc free reflection}. The proof for $c^{-r}\gamma_i$ is similar.
\end{proof}

Next we consider real roots with finite $c$-orbits. When $n = 2$ (affine type means $c_{12}c_{21}=4$), by basic linear algebra there are no finite $c$-orbits on real roots. Now suppose that $n\geq 3$. By a \emph{tube}, we simply mean a set $T = \mathbb Z_{\geq 1} \times \mathbb Z/d\mathbb Z$ for some $d\in \mathbb Z_{\geq 1}$. It is said to have \emph{period} $d$. An element $(n, m)\in T$ is said to be on level $n$.

\begin{proposition}[\cite{MR0447344}, \cite{MR4099768}]\label{prop: finite c orbit}
    Any finite $c$-orbit in $\Delta_\mathrm{re}$ contains either only positive roots or negative roots. Positive real roots with finite $c$-orbits are $\beta_{(n,m)}^{(i)}$ parametrized by the elements 
    \[
        \{(n, m)\mid d_i \nmid n \}\subseteq T_i = \mathbb Z_{\geq 1} \times \mathbb Z/d_i\mathbb Z
    \]
    in finitely many tubes $T_i$ where $i = 1, \dots, \ell$, satisfying mesh relations
    \begin{equation}\label{eq: mesh relation}
        \beta_{(n, m)}^{(i)} + \beta_{(n, m+1)}^{(i)} = \beta_{(n+1, m)}^{(i)} + \beta_{(n-1, m+1)}^{(i)},
    \end{equation}
    where $\beta_{(0,m)}^{(i)}$ is set to $0$ for any $m$ and $i$.
    The Coxeter element $c$ acts on each tube by
    \begin{equation}\label{eq: c action on tube}
        c(\beta_{(n, m)}^{(i)}) = \beta_{(n, m + 1)}^{(i)}.
    \end{equation}
\end{proposition}

Let us elaborate on the above proposition. Reading and Stella \cite{MR4099768} provide a precise linear algebraic description of roots in these tubes as follows. The Coxeter element $c$, acting on the complexification $\mathbb{C}^I$ of $\mathbb{R}^I$, has a set of $n-1$ linearly independent eigenvectors. They span a hyperplane $U^c \subset \mathbb {C}^I$. We choose an extended vertex $k \in I$ as in \cite{MR4099768}; see Table 1 in \emph{loc. cit.} In particular, the Cartan sub-matrix with indices $I\setminus \{ k \}$ is of finite type. Denote the corresponding finite root system by $\Phi_\mathrm{fin}$ (which is naturally embedded in $\Phi$, the root system of $C$).

Let $\Upsilon^c \coloneqq \Phi \cap U^c$ and $\Upsilon_\mathrm{fin}^c \coloneqq \Phi_\mathrm{fin} \cap U^c$. The later is shown in \cite{MR4099768} to be a finite root system of rank $n-2$ (in fact as a product of at most three type $A$ root systems). Denote by $\Xi_\mathrm{fin}^c$ the set of positive (i.e. in $\Phi^+_\mathrm{fin}$) simple roots for $\Upsilon_\mathrm{fin}^c$.

\begin{proposition}[\cite{MR4099768}]\label{prop: bottom of tube} 
    Let $\Phi$ be an affine root system and $c$ a Coxeter element.
        \begin{enumerate}
            \item The finite root system $\Upsilon_\mathrm{fin}^c$ is of type $A_{d_1-1} \times \cdots \times A_{d_\ell - 1}$ where $\ell \leq 3$.
            \item The bottom of each tube has exactly $d_i - 1$ simple roots in $\Xi_\mathrm{fin}^c$ corresponding to some root sub-system $A_{d_i - 1}$.
        \end{enumerate}
\end{proposition}

The roots $\beta_{(n, m)}^{(i)}$ on the tubes $T_i$ in \Cref{prop: finite c orbit} can then be described as follows. Let $\alpha_1^{(i)},\dots, \alpha_{d_i-1}^{(i)}$ denote the simple roots in $A_{d_i - 1}$ for $i = 1, \dots, \ell$ so that they form the Dynkin diagram 
\[
    1 \frac{}{\quad} 2 \frac{}{\quad} \dots \frac{}{\quad} (d_{i}-1).
\]
In fact we can set
\[
    \beta^{(i)}_{(1, m)} = \alpha_{m}^{(i)},\quad m = 1, \dots, d_i-1.
\]
Notice that $\beta^{(i)}_{(1, 0)} = c(\alpha_{d_i - 1}^{(i)}) = c^{-1}(\alpha_{1}^{(i)})$ by (\ref{eq: c action on tube}) in \Cref{prop: finite c orbit}. Then one can use the mesh relation to express any root on the tube of higher level as a linear sum of the roots at the \emph{bottom} of $T_i$ (i.e. of level $1$).

\begin{proposition}\label{prop: real schur roots affine type}
    The real Schur roots $\Delta_{\mathrm{rS}}(C, \Omega)$ is a disjoint union of $\{c^r\beta_i\mid i\in I, r\in \mathbb N\}$, $\{c^{-r}\gamma_i\mid i\in I, r\in \mathbb N\}$, and
    \[
        \bigsqcup_{i = 1}^\ell \{\beta_{(n, m)}^{(i)} \mid (n, m)\in T_i, n\leq d_i - 1\}.
    \]
\end{proposition}

\begin{proof}
    This description is built on a sequence of results. Linear algebraically, this can be obtained combining \cite{MR4099768} and \cite{MR4073889}. Alternatively we can use the categorical description of all real Schur roots (as defined in (\ref{eq: real schur roots})) by Hubery and Krause \cite{MR3551191}. In fact, they show that real Schur roots are exactly the dimension vectors of rigid indecomposable modules of any hereditary algebra (over an arbitrary field $K$) whose associated \emph{generalized Cartan lattice} (see \cite[Section 3]{MR3551191}) is given by the data $(C, D, \Omega)$. Notice that we are in the affine case. Then one can take the explicit construction of a hereditary algebra of type $(C, D, \Omega)$ in \cite{MR0447344}. Then the rigid indecomposable modules have dimension vectors exactly as described in the statement according to \cite{MR0447344}.
\end{proof}

Notice that there is at least one positive root in $\Upsilon^c_\mathrm{fin}$ on each level of the tube $T_i$ up to $d_i - 1$. In fact, these roots are
\begin{equation}\label{eq: finite real schur}
    \beta_{(b-a+1, a)}^{(i)}  =  \sum_{m\in [a, b]} \alpha_{m}^{(i)} \in \Phi_\mathrm{fin}^+
\end{equation}
for a sub-interval $[a, b]\subset [1, d_i-1]$. For example $\beta^{(i)}_{(d_i-1, 1)} = \sum_{m=1}^{d_i-1} \alpha_{m}^{(i)}$, the longest positive root in the root system $A_{d_i-1}$, and it is the only one in $\Upsilon^c_\mathrm{fin}$ on the level $d_i - 1$.

We next realize the real Schur roots with finite $c$-orbits by rank vectors of $\tau$-locally free $H$-modules. The proposition below is a direct consequence of \Cref{thm: real schur roots rigid module} and \Cref{prop: reflection of loc free and rigid,prop: finite c orbit,prop: real schur roots affine type}. We provide an alternative proof that emphasizes the finite type GLS algebra $H_\mathrm{fin}$ (see below).

\begin{proposition}\label{prop: tau rigid modules on tubes}
    There are indecomposable rigid locally free $H$-modules $M_{(n,m)}^{(i)}$ with $n\leq d_i - 1$ such that
    \[
        \beta_{(n,m)}^{(i)} = \rank M_{(n,m)}^{(i)} \quad \text{and} \quad \tau M_{(n,m)}^{(i)} = M_{(n,m-1)}^{(i)}.
    \]
\end{proposition}

\begin{proof}
    In fact, we can start by constructing those modules whose rank vectors are in $\Phi_\mathrm{fin}^+$. Denote by $H_\mathrm{fin}$ the GLS algebra associated to the Cartan submatrix with indices $I \setminus \{k\}$ where $k$ is the chosen extended vertex. Then by \cite[Theorem 1.3]{MR3660306}, each $\beta_{(b-a+1, a)}^{(i)}$ as in (\ref{eq: finite real schur}), which is in $\Phi_\mathrm{fin}^+$, is the rank vector of a unique indecomposable rigid locally free $H_\mathrm{fin}$-module. These $H_\mathrm{fin}$-modules are also locally free and rigid as $H$-modules, hence $\tau$-locally free by \Cref{prop: rigid implies tau loc free}. Then by \Cref{prop: tau loc free reflection}, we can use AR translation $\tau$ or the Coxeter functors to transport these modules to fill in the tubes up to level $d_i - 1$.
\end{proof}

\begin{remark}
    The modules $M_{(1,m)}^{(i)}$ are explicitly constructed and proven to be the \emph{mouth} modules of actual tubes (as components of the Auslander--Reiten quiver) of $H$-modules in \cite{lin2024affine}.
\end{remark}

\begin{example}\label{ex: affine b3}
    We present an example in type $\widetilde {\mathbb B}_3$. Let the data $(C, D, \Omega)$ be
    \[
        C = \begin{bmatrix}
            2 & -2 & 0 & 0 \\
            -1 & 2 & -1 & 0 \\
            0 & -1 & 2 & -1 \\
            0 & 0 & -2 & 2
        \end{bmatrix},
        \quad
        D = \begin{bmatrix}
            1 & 0 & 0 & 0 \\
            0 & 2 & 0 & 0 \\
            0 & 0 & 2 & 0 \\
            0 & 0 & 0 & 1
        \end{bmatrix},
        \quad \text{and} \quad
        \Omega = \{(3, 4), (2, 3), (1, 2)\}.
    \]
    The algebra $H = H_K(C, D, \Omega)$ can be expressed as $KQ/I$, where
    \[
        Q = \begin{tikzcd}
            4 \ar[r] & 3 \ar[loop, in=60, out=120, distance=6mm, "\varepsilon_3"] \ar[r, "\alpha"] & 2 \ar[loop, in=60, out=120, distance=6mm, "\varepsilon_2"] \ar[r] & 1
        \end{tikzcd} \quad \text{and} \quad
        I = ( \varepsilon_2^2,\ \varepsilon_3^2,\ \varepsilon_2 \alpha - \alpha \varepsilon_3 ).
    \]
    In the following we represent $H$-modules as quiver representations. Each copy of $i\in \{1, 2, 3, 4\}$ in a diagram below stands for a basis vector in $M_i$ of a representation $M$ of $Q$. Any arrow with a horizontal shift means that the corresponding arrow in $Q$ transports the basis vectors. Vertical arrows express the action of $\varepsilon_i$.

    Choose the vertex $4$ to be the extended vertex. In this case the finite root system $\Upsilon_\mathrm{fin}^c$ is of type $A_{2}$ with two (positive) simple roots $\alpha_2 = (0, 1, 0, 0)$ and $\alpha_3 = (0, 0, 1, 0)$. They correspond to pseudo-simples $E_2$ and $E_3$ in \Cref{prop: tau rigid modules on tubes}. The finite $\tau$-orbit of $E_2$ and $E_3$ is
    \[
        E_3 \quad \xlongleftarrow{\tau} \quad 
        \begin{tikzcd}
            4 \ar[r] & 3 \ar[r] \ar[d] & 2 \ar[r] \ar[d] & 1 \\
            4 \ar[r] & 3 \ar[r] & 2 \ar[r] & 1 
        \end{tikzcd}
        \quad \xlongleftarrow{\tau} \quad E_2 \quad \xlongleftarrow{\tau} \quad E_3.
    \]
    The level $2$ of this tube (of roots) is another finite $c$-orbit
    \[
        \beta_{(2, 1)}^{(1)} = (0, 1, 1, 0) \quad \xlongrightarrow{c} \quad \beta_{(2, 2)}^{(1)} = (2, 1, 2, 2)
        \quad \xlongrightarrow{c} \quad \beta_{(2, 0)}^{(1)} = (2, 2, 1, 2) \quad \xlongrightarrow{c} \quad \beta_{(2, 1)}^{(1)}.
    \]
    The corresponding locally free and rigid indecomposable modules are
    
    \[  
        M_{(2, 1)}^{(1)} = \begin{tikzcd}
            3 \ar[r] \ar[d] & 2 \ar[d] \\
            3 \ar[r] & 2 
        \end{tikzcd},
        \quad
        M_{(2, 2)}^{(1)} =
        \begin{tikzcd}
            4 \ar[r] & 3 \ar[r] \ar[d] & 2 \ar[r] \ar[d] & 1 \\
            4 \ar[r] \ar[rd] & 3 \ar[r] & 2 \ar[r] & 1 \\
             & 3 \ar[d] & & \\
             & 3 & & 
        \end{tikzcd}, 
        \quad
        M_{(2, 0)}^{(1)} = 
        \begin{tikzcd}
            4 \ar[r] & 3 \ar[r] \ar[d] & 2 \ar[r] \ar[d] & 1 \\
            4 \ar[r] & 3 \ar[r] & 2 \ar[r] & 1 \\
             & & 2 \ar[d] & \\
             & & 2 \ar[uur] & 
        \end{tikzcd}.
    \]
    They form a finite $\tau$-orbit as
    \[
        M_{(2, 1)}^{(1)} \quad \xlongleftarrow{\tau} \quad M_{(2, 2)}^{(1)} \quad \xlongleftarrow{\tau} \quad M_{(2, 0)}^{(1)} \quad \xlongleftarrow{\tau} \quad M_{(2, 1)}^{(1)}.
    \]
\end{example}

\section{Cluster algebras}\label{section: cluster and coxeter}

\subsection{Cluster algebras}
We review the Fomin--Zelevinsky cluster algebra $\mathcal A(\widetilde B)$ \cite{MR1887642} associated to $\widetilde B\in \mathrm{Mat}_{m\times n}(\mathbb Z)$ (with $m\geq n$) extended from a skew-symmetrizable matrix $B\in \mathrm{Mat}_{n\times n}(\mathbb Z)$, that is, the first $n$ rows of $\widetilde B$ form $B$. For the original systematic treatment, we refer to \cite{MR2295199}.

Let $\mathbb T_n$ denote the infinite $n$-regular tree. The $n$-edges incident to a vertex are distinctively labeled by indices $I = \{1, \dots, n\}$. By assigning the \emph{initial seed} $\Sigma = (\widetilde B, (x_1, \dots, x_n, x_{n+1}, \dots, x_m))$ to a root $t_0\in \mathbb T_n$ and applying \emph{seed mutations} $\mu_k$ for $k = 1, \dots, n$, one obtain an assignment of a seed to each $t\in \mathbb T_n$,
\[
    t\mapsto \Sigma_t = (\widetilde B_t, (x_{1;t}, \dots, x_{m;t}))\in \mathrm{Mat}_{m\times n}(\mathbb Z) \times \mathbb Q(x_1, \dots, x_m)^m.
\]
This assignment is uniquely determined to satisfy $\mu_k(\Sigma_t) = \Sigma_{t'}$ for any edge $t \frac{k}{\quad\quad} t'$ in $\mathbb T_n$ which means if writing $\widetilde B_t = (b^t_{ij})$ for $t\in \mathbb T_n$, then
\[
    b^{t'}_{ij} = \begin{dcases}
        - b^t_{ij} \quad & \text{if $i = k$ or $j = k$} \\
        b^t_{ij} + \mathrm{sgn}(b^t_{ik})[b^t_{ik}b^t_{kj}]_+ \quad & \text{otherwise,}
    \end{dcases}
\]
and
\[
    x_{i; t'} = \begin{dcases}
        x_{k; t}^{-1}\left(\prod_{j=1}^m x_{j;t}^{[b^t_{jk}]_+} + \prod_{j=1}^m x_{j;t}^{[-b^t_{jk}]_+} \right) \quad & \text{if $i = k$}\\
        x_{i; t} \quad & \text{otherwise.}
    \end{dcases}
\]
Each $x_{i; t}$ for $i = 1, \dots, n$ is called a (mutable) \emph{cluster variable}. A \emph{cluster monomial} is a monomial of $\{x_{1;t}, \dots, x_{n; t}\}$ for some $t\in \mathbb T_n$. Notice that according to the mutation rule, the variables $x_{n+1; t},\dots, x_{m; t}$ stay as the initial ones. They are thus called \emph{frozen variables}, which can be viewed as playing the role of coefficients.

\begin{definition}\label{def: cluster algebra}
    The \emph{cluster algebra} $\mathcal A(\widetilde B)$ is defined to be the subalgebra of $\mathbb Q(x_1, \dots, x_m)$ over $\mathbb Z[x_{n+1}^\pm, \dots, x_{m}^\pm]$ generated by all $x_{i; t}$.
\end{definition}

When there is need to emphasize the dependence on the root $t_0$, we shall denote $x_{i; t} = x_{i; t}^{\widetilde B; t_0}$.

\begin{remark}
    When $\widetilde B = B$, we call $\mathcal A(B)$ \emph{coefficient-free}. When $\widetilde B = \begin{bsmallmatrix}
        B \\
        I
    \end{bsmallmatrix}$ where $I$ is the identity matrix of dimension $d$, we say that $\mathcal A(\widetilde B)$ has \emph{principal coefficients} at $t_0$. In this case, we denote $\mathcal A_\bullet(B) = \mathcal A(\widetilde B)$.
\end{remark}

Already remarkably, each $x_{i; t}$ actually belongs to $\mathbb Z[x_1^\pm, \dots, x_n^\pm, x_{n+1}, \dots, x_m]$ (the Laurent phenomenon \cite{MR1887642}). This allows us to define the $\mathbf d$-vector of a cluster variable.

\begin{definition}
    The \emph{$\mathbf d$-vector} $\mathbf d(x) = (d_1, \dots, d_n)\in \mathbb Z^n$ of a cluster variable $x$ (with respect to the initial seed) is the minimal vector in $\mathbb Z^n$ such that
    $x\prod_{i = 1}^n x_i^{d_i}$ is a polynomial in $x_1, \dots, x_m$.
\end{definition}

Let $B = B(C, \Omega)$ be the skew-symmetrizable matrix as in \Cref{def: matrix b}. In particular it is acyclic; but we do not restrict to the affine case. The following theorem accumulates on a series of well-known results in the additive categorification of acyclic cluster algebras.

\begin{theorem}\label{thm: d vectors real schur roots}
    Sending a cluster variable $x$ to its $\mathbf d$-vector $\mathbf d(x)$ induces a bijection from the set of non-initial cluster variables in $\mathcal A_\bullet(B)$ to set $\Delta_\mathrm{rS}(C, \Omega)$ of real Schur roots.
\end{theorem}

\begin{proof}
    Recall that we have defined real Schur roots using the data $(C, D, \Omega)$ in \Cref{subsec: real schur roots}. Hubery and Krause \cite{MR3551191} showed that they are in bijection with exceptional modules (as their dimension vectors) of an hereditary algebra (over any ground field) of type $(C, D, \Omega)$. 
    
    For $C$ symmetric of Dynkin type, one can take $H$ to be $H_\mathbb C(C, I, \Omega)$ which is the path algebra of a Dynkin quiver. Then the result follows from Caldero--Chapoton's formula \cite{MR2250855} expressing cluster variables from exceptional $H$-modules where the $\mathbf d$-vector is evidently the dimension vector. This formula has then been extended by Caldero and Keller \cite{MR2316979} to any symmetric $C$.

    For $C$ symmetrizable, the statement follows from Rupel's quantum cluster character formula \cite{MR3378824}. In this case, one can construct $H$ a hereditary algebra of type $(C, D, \Omega)$ over finite fields. And again through the formula, the dimension vector of an exceptional module is exactly the $\mathbf d$-vector of the corresponding cluster variable.
\end{proof}

Recall from \Cref{thm: real schur roots rigid module} that real Schur roots are in bijection with indecomposable locally free rigid $H_\mathbb C(C, D, \Omega)$-modules (as their rank vectors). Together with the bijection in \Cref{thm: d vectors real schur roots}, this induces a bijection between indecomposable locally free rigid modules and non-initial cluster variables. For $\beta\in \Delta_\mathrm{rS}(C, \Omega)$, we denote the corresponding $H$-module and cluster variable by
\begin{equation}
    \begin{tikzcd}
        M(\beta) \ar[r] & X_\beta \ar[l]
    \end{tikzcd}.
\end{equation}

Our focus in this paper is the case when $C$ is of affine type. We will show that $X_\beta$ can be expressed by a Caldero--Chapoton type formula from $M(\beta)$, providing a new cluster character formula in this skew-symmetrizable case different from those of Rupel \cite{MR3378824} and of Demonet \cite{MR2844757}. An advantage of using $H$-modules is that certain classical constructions of $\mathbb C$-varieties such as quiver Grassmannians carry over with modifications.

\subsection{\texorpdfstring{$F$-polynomials and $\mathbf g$-vectors}{F-polynomials and g-vectors}}\label{subsec: g vector f poly}

With principal coefficients, cluster variables in $\mathcal A_\bullet (B)$ are determined by their $F$-polynomials and $\mathbf g$-vectors as shown in \cite{MR2295199}, which we mostly follow in this section. They are defined as follows.

We denote $x_{n+1}, \dots, x_{2n}$ by $y_1, \dots, y_n$. The Laurent phenomenon allows the following definition.

\begin{definition}\label{def: F-poly}
    The $F$-polynomial $F_{i; t}^{B; t_0}$ of $x_{i; t}$ (with respect to the root $t_0$) is defined as
    \[
         F_{i; t} = F_{i; t}^{B; t_0}(y_1, \dots, y_n) \coloneqq x_{i; t}(1, \dots, 1, y_1, \dots, y_n)\in \mathbb Z[y_1, \dots, y_n].
    \]
\end{definition}

For example, the $F$-polynomial of any initial cluster variable $x_{i; t_0} = x_i$ is the constant $1$.

There is a $\mathbb Z^n$-grading on $\mathcal A_\bullet(B)$. Let $e_1, \dots, e_n$ be the standard basis of $\mathbb Z^n$. We define 
\[
    \deg x_i \coloneqq e_i \quad \text{and} \quad \deg y_i \coloneqq \sum_{j=1}^n -b_{ji}e_j.
\]
This gives a $\mathbb Z^n$-grading on $\mathbb Z[x_1^\pm,\dots, x_n^\pm, y_1, \dots, y_n]$.

\begin{proposition}[\cite{MR2295199}]
    Every cluster variable in $\mathcal A_\bullet(B)$ is homogeneous.  We thus can define the $\mathbf g$-vector of a cluster variable $x_{i; t}$ to be
    \[
        \mathbf g_{i; t} = \mathbf g_{i; t}^{B; t_0} \coloneqq \deg x_{i; t} \in \mathbb Z^n.
    \]
\end{proposition}

\begin{theorem}[\cite{MR2295199}]
    Every cluster variable in $\mathcal A_\bullet(B)$ can be expressed as
    \[
        x_{k; t} =  F_{k; t}(\hat y_1, \dots, \hat y_n) \prod_{i=1}^nx_i^{g_i} = F_{k; t}(\hat y_1, \dots, \hat y_n) x^{\mathbf g_{k; t}}
    \]
    where $\mathbf g_{k; t} = (g_i)_{i\in I}$ and $\hat y_i = y_i\prod_{j\in I} x_j^{b_{ji}}$.
\end{theorem}

The significance of $F$-polynomial and $\mathbf g$-vector is in the following \emph{separation formula}, which gives an expression of any cluster variable in $\mathcal A(\widetilde B)$, where $\widetilde B$ is extended from $B$. Let $\mathbb Q_\mathrm{sf}(y_1, \dots, y_n)$ denote the \emph{universal semifield} of subtraction free rational expressions on $n$ variables over $\mathbb Q$. The addition and multiplication are the usual ones of rational functions $\mathbb Q(y_1, \dots, y_n)$. Any $F_{i;t}$ is in this semifield because any cluster variable is obtained by iterative mutations only involving additions and divisions. For any $F\in \mathbb Q_\mathrm{sf}(y_1, \dots, y_n)$ and any semifield $\mathbb P$, denote by
\[
    F\mid_\mathbb P (p_i \leftarrow y_i) \in \mathbb P
\]
the evaluation of $F$ at $y_i = p_i\in \mathbb P$.

\begin{theorem}[{\cite[Corollary 6.3]{MR2295199}, \cite[(2.14)]{MR2629987}}]\label{thm: separation formula}
    Let $\widetilde B\in \mathrm{Mat}_{m\times n}(\mathbb Z)$ be any matrix extended from $B$. Then we have
    \[
        x_{\ell; t}^{\widetilde B; t_0} = \frac{F_{\ell; t}(\hat y_1, \dots, \hat y_n)}{F_{\ell; t} \mid_{\mathrm{Trop}(x_{n+1}, \dots, x_{m})} \left(\prod_{j=n+1}^{m} x_{j}^{\tilde b_{ji}} \leftarrow y_i \right)} \prod_{i=1}^n x_i^{g_i}
    \]
    where $\hat y_i = \prod_{j=1}^m x_j^{\tilde b_{ji}}$ and $\mathrm{Trop}(x_{n+1}, \dots, x_m)$ denotes the \emph{tropical semifield} generated by indeterminates $x_{n+1}$, \dots, $x_m$ (see \cite[Definition 2.1, 2.2]{MR2295199}).
\end{theorem}

\begin{remark}\label{rmk: extended g vector trop eva}
    The tropical evaluation of $F_{\ell;t}$ results in a Laurent monomial of $x_{n+1}, \dots, x_{m}$ with non-positive exponents. Thus the formula in \Cref{thm: separation formula} can be written as
    \[
        x_{\ell; t}^{\widetilde B; t_0} = F_{\ell; t}(\hat y_1, \dots, \hat y_n) \prod_{i = 1}^m {x_i^{g_i}}, 
    \]
    where $\widetilde {\mathbf g}_{\ell; t} = (g_i)_{i=1}^m$ is called the \emph{extended} $\mathbf g$-vector (which depends on $\widetilde B$).
\end{remark}

\begin{definition}\label{def: h vector fz}
    Define $\mathbf h_{\ell; t}^{B;t_0} = (h_i)_{i\in I}$ by
    \[
        x_1^{h_1}\cdots x_n^{h_n} = F_{\ell;t}^{B; t_0} \mid_{\mathrm{Trop}(x_1, \dots, x_n)}(x_i^{-1}\prod_{j = 1}^n x_j^{[-b_{ji}]_+} \leftarrow y_i ).
    \]
\end{definition}

Last we review a recurrence of $\mathbf g$-vectors and $F$-polynomials under the change of initial seeds from \cite{MR2295199}.

\begin{proposition}[{\cite{MR2295199}, \cite[Proposition 2.4]{MR2629987}}]\label{prop: recurrence g vector f poly dwz}
    Suppose $t_0 \frac{k}{\quad\quad} t_1$ in $\mathbb T_n$ and $B_1 = \mu_k(B)$. Let $h_k$ (resp. $h_k'$) be the $k$-th component of $\mathbf h_{\ell;t}^{B; t_0}$ (resp. $\mathbf h_{\ell;t}^{B_1; t_1}$). Then the $\mathbf g$-vectors $\mathbf g_{\ell;t}^{B; t_0} = (g_i)_{i=1}^n$ and $\mathbf g_{\ell; t}^{B_1; t_1} = (g_i')_{i=1}^n$ are related by
    \begin{equation}\label{eq: recurrence g vector dwz}
        g_k' = - g_k \quad \text{and} \quad g_i' = g_i + [b_{ik}]_+g_k - b_{ik}h_k \quad \text{for $i\neq k$}.
    \end{equation}
    We also have
    \begin{equation}\label{eq: hk and hk'}
        g_k = h_k - h_k' \quad \text{and}
    \end{equation}
    \begin{equation}\label{eq: recurrence f poly dwz}
            (1 + y_k)^{h_k} F_{\ell;t}^{B; t_0}(y_1,\dots, y_n) = (1 + y_k')^{h_k'}F_{\ell;t}^{B; t_1}(y_1', \dots, y_n').
    \end{equation}
    where $y_k' = y_k^{-1}$ and $y_i' = y_iy_k^{[b_{ki}]_+}(y_k+1)^{-b_{ki}}$ for $i\neq k$.
\end{proposition}

\section{Caldero--Chapoton functions}\label{section: cc function}

\subsection{\texorpdfstring{$F$-polynomials and CC functions}{F-polynomials and CC functions}}
Fix the ground field $K = \mathbb C$ and let $H = H_\mathbb C(C, D, \Omega)$ as in \Cref{subsec: def of h}. For $M\in \replf H$ and $\mathbf e\in \mathbb N^n$, the \emph{locally free quiver Grassmannian} is
\[
    \gr_\mathbf e^H(M) \coloneqq \{ N \mid \text{$N\subset M$ locally free and $\rank N = \mathbf e$} \}.
\]
It has a structure of quasi-projective (complex) variety, which will become clear in \Cref{section: generic bases}. Denote by $\chi(\cdot)$ the Euler characteristic in analytic topology.

\begin{definition}[\cite{MR3830892}]\label{def: locally free cc function}
    For $M\in \replf H$ with $\rank M = \mathbf m = (m_i)_{i\in I}$, the \emph{$F$-polynomial} of $M$ is
    \[
        F_M^H = F_M(y_1, \dots, y_n) \coloneqq \sum_{\mathbf r \in \mathbb N^n} \chi(\gr_\mathbf r^H(M)) \prod_{i = 1}^n y_i^{r_i} \in \mathbb Z[y_1, \dots y_n].
    \]
    The \emph{(locally free) Caldero--Chapoton function with principal coefficients} associated to $M$ is
    \[
        X_M^\bullet = \prod_{i=1}^n x_i^{-m_i} \cdot \sum_{\mathbf r\in \mathbb N^n} \chi(\gr_\mathbf r^H(M))\prod_{i=1}^n y_i^{r_i} \prod_{i=1}^n x_i^{\sum_{j=1}^n [-b_{ij}]_+m_j + b_{ij}r_j} \in \mathbb Z[y_1, \dots, y_n, x_1^\pm, \dots, x_n^\pm].
    \]
\end{definition}

The coefficient-free CC function $X_M$ is defined by evaluating $X_M^\bullet$ at $y_i = 1$. Notice that $\prod x_i^{-m_i}$ is the minimal denominator of $X_M$ as a Laurent polynomial; see for example \cite[Remark 4.5]{MR4732400}. We thus call $(m_i)_{i\in I}\in \mathbb N^n$ the $\mathbf d$-vector of $X_M$ (also of $X_M^\bullet)$.

\begin{remark}\label{rmk: cc function direct sum}
    It is clear from \cite{MR3660306} that for any $M, N\in \replf H$, we have 
    \[
        F_{M\oplus N} = F_M \cdot F_N \quad \text{and} \quad X_{M\oplus N} = X_M \cdot X_N,
    \]
    which holds for any $C$.
\end{remark}

\begin{remark}
    Using the injective $\mathbf g$-vector $\mathbf g_H^\mathrm{inj}(M) = (g_i)_{i\in I}$ (see \Cref{def: inj and proj g-vector}) and \Cref{lemma: g vector from rank vector}, the CC function $X_M^\bullet$ can also be expressed as
    \[
        X_M^\bullet =  F_M(\hat y_1, \dots, \hat y_n) \prod_{i = 1}^n x_i^{g_i}
    \]
    where $\hat y_i = y_i \prod_{j = 1}^{n} x_j^{b_{ji}}$.
\end{remark}

We define another vector $\mathbf h(M)\in \mathbb Z^n$ following \cite{MR2629987}. Recall the notations in \Cref{subsec: g vector f poly}.

\begin{definition}\label{def: h vector}
    Under the assumption that $F_M$ is in $\mathbb Q_\mathrm{sf}(y_1, \dots, y_n)$, we define $\mathbf h(M) = (h_i)_{i\in I}$ by
    \[
        x_1^{h_1}\cdots x_n^{h_n} = F_M \mid_{\mathrm{Trop}(x_1, \dots, x_n)}(x_i^{-1}\prod_{j\in I} x_j^{[-b_{ji}]_+} \leftarrow y_i ).
    \]
\end{definition}

\begin{lemma}\label{lemma: h vector}
    If $k\in I$ is a sink, then $h_k = -m_k$. If $k$ is a source and the map $M_{k, \mathrm{out}}$ is injective, then $h_k = 0$.
\end{lemma}

\begin{proof}
    This exercise follows from the idea in \cite[Proposition 3.3]{MR2629987}. In fact, in the tropical evaluation of \Cref{def: h vector}, we can replace $F_M(y_1, \dots, y_n)$ with a sum of (coefficient-one) monomials with exponents being the vertices of the Newton polytope of $F_M$ (the convex hull in $\mathbb R^n$ of the exponents of monomials in $F_M$). The exponent of $x_k$ in the substitution
    \[
        y_1^{r_1}\dots y_n^{r_n} \mid (x_i^{-1}\prod_{j\in I} x_j^{[-b_{ji}]_+} \leftarrow y_i )
    \]
    reads as $-r_k + \sum_{i\in I}[-b_{ki}]_+r_i$. Then $h_k$ is the minimum of such from $(r_i)_{i\in I}$ the rank vectors of locally free submodules representing Newton polytope vertices.
    
    If $k$ is a sink, then $h_k = -m_k$ is attained by taking the submodule $E_k^{m_k}$. If $k$ is a source such that $M_{k, \mathrm{out}}$ is injective, the restriction of $M_{k, \mathrm{out}}$ on any submodule is again injective. Then for any submodule with rank $(r_i)_{i\in I}$, we have $-r_k + \sum_{i\in I}[-b_{ki}]_+r_i \geq 0$. It follows that $h_k = 0$ by taking the zero submodule.
\end{proof}

\subsection{Recurrence under reflections}

We recall the following recurrence of locally free $F$-polynomials under reflections.

Let $k\in I$. The following proposition concerns with $M\in \replf H$ satisfying the property that
\[
    \text{the map $M_{k, \mathrm{in}}$ is surjective when $k$ is a sink or $M_{k, \mathrm{out}}$ is injective when $k$ is a source ($\star$)}.
\]

\begin{proposition}[{\cite[Prop 4.7 and Cor 4.8]{MR4732400}}]\label{prop: recurrence under reflection uniform ver}
     Let $k$ be either a sink or a source. Let $M\in \replf H$ be of rank $\mathbf m = (m_i)_{i\in I}$ satisfying $(\star)$. Let $h_k(M) = -m_k$ when $k$ is a sink and $h_k(M) = 0$ when $k$ is a source. Denote by $M'\coloneqq F_k^\pm(M)\in \rep s_k(H)$ the reflection of $M$ at $k$.
    \begin{enumerate}
        \item The module $M'$ is again locally free satisfying $(\star)$ and its rank vector is $\mathbf m' = s_{\alpha_k}(\mathbf m)$.
        
        \item The $F$-polynomials satisfy the equation
        \begin{equation}\label{eq: recurrence f poly uni ver}
            (1 + y_k)^{h_k(M)} F_M^H(y_1,\dots, y_n) = (1 + y_k')^{h_k(M')}F_{M'}^{s_k(H)}(y_1', \dots, y_n').
        \end{equation}
        where $y_i' = y_iy_k^{[b_{ki}]_+}(y_k+1)^{-b_{ki}}$ for $i\neq k$ and $y_k' = y_k^{-1}$.
        \item The (injective) $\mathbf g$-vectors
        \[
            \mathbf g^\mathrm{inj}_H(M) = (g_i)_{i\in I} \quad \text{and} \quad \mathbf g^\mathrm{inj}_{s_k(H)}(M') = (g'_i)_{i\in I}
        \]
        are related by
        \begin{equation}\label{eq: recurrence g vector uni ver}
             g_k' = - g_k \quad \text{and} \quad g_i' = g_i + [b_{ik}]_+g_k - b_{ik}h_k(M) \quad \text{for $i\neq k$}.
        \end{equation}
    \end{enumerate}
\end{proposition}

\begin{remark}\label{rmk: h vectors}
    In the situation of \Cref{prop: recurrence under reflection uniform ver}, we have
    \[
        h_k(M)h_k(M') = 0 \quad \text{and} \quad g_k = h_k(M) - h_k(M').
    \]
\end{remark}

\subsection{Cluster variables as CC functions}\label{subsec: cluster variables cc function}

This section gives the proof of our first main theorem \Cref{thm: first main theorem intro}, that is, in the affine case for any given $\beta\in \Delta_\mathrm{rS}$, the CC function $X_{M(\beta)}$ of the locally free indecomposable rigid module $M(\beta)$ (whose rank vector is $\beta$) equals the cluster variable whose $\mathbf d$-vector is $\beta$.

We start with the roots in an infinite $c$-orbit described in \Cref{prop: infinite c orbits}.

First let $\beta = \beta_\ell = s_1s_2\dots s_{\ell-1}(\alpha_\ell)$ for $\ell\in I$. Consider in $\mathbb T_n$ the following path
\begin{equation}\label{eq: sink sequence}
    \begin{tikzcd}
        t_0 \ar[r, dash, "1"] & t_1 \ar[r, dash, "2"] & \cdots \ar[r, dash, "n-1"] & t_{n-1} \ar[r, dash, "n"] & t_n.
    \end{tikzcd}
\end{equation}
Let $(C, D, \Omega)$ be associated to $t_0$ and subsequently reflect $\Omega$ along the path, that is, let $\Omega(t_i) \coloneqq s_i\dots s_1(\Omega)$ be associated with $t_i$ for $i\in I$. Thus we denote $H(t_i) = H(C, D, \Omega(t_i))$. Let $E_\ell(t_i)\in \rep H(t_i)$ denote the pseudo-simple module of rank $\alpha_\ell$. Let $P_{\ell}( t_i) = H(t_i)e_\ell \in \rep H(t_i)$ denote the projective cover of $E_\ell(t_i)$.

\begin{lemma}\label{lemma: f poly g vec projective}
    For any $\ell\in I$, we have
    \[
        \rank P_\ell(t_0) = \beta_\ell,\quad F_{\ell; t_n}^{B; t_0} = F_{P_{\ell}(t_0)} \quad \text{and} \quad \mathbf g_{\ell; t_n}^{B; t_0} = \mathbf g^{\mathrm{inj}}(P_\ell(t_0)).
    \]
\end{lemma}

\begin{proof}
    We first note that $B_{t_n} = B_{t_0} = B$. By definition, we have
    \[
        F_{\ell; t_n}^{B_{t_n}; t_n} = \dots = F_{\ell; t_n}^{B_{t_\ell}; t_\ell} = 1
        \quad \text{and} \quad
        \mathbf g_{\ell; t_n}^{B_{t_n}; t_n} = \dots = \mathbf g_{\ell; t_n}^{B_{t_\ell}; t_\ell} = e_\ell.
    \]
    Consider the mutation on the edge $t_{\ell-1} \frac{\ell}{\quad\quad} t_\ell$. It is clear that
    \[
        F_{\ell; t_n}^{B_{t_{\ell-1}}; t_{\ell-1}} = 1 + y_\ell \quad \text{and} \quad
        \mathbf g_{\ell; t_n}^{B_{t_{\ell-1}}; t_{\ell-1}} = -e_\ell + \sum_{i\in I} [-b_{i,\ell}^{t_{\ell-1}}]_+e_i.
    \]
    Notice that on $t_{\ell-1}\in \mathbb T_n$, the vertex $\ell$ is a sink of $\Omega(t_{\ell-1})$. Thus we have
    \begin{equation}\label{eq: base case f poly g vec}
        F_{\ell; t_n}^{B_{t_{\ell-1}}; t_{\ell-1}} = F_{E_{\ell}(t_{\ell-1})} \quad \text{and} \quad
        \mathbf g_{\ell; t_n}^{B_{t_{\ell-1}}; t_{\ell-1}} = \mathbf g^\mathrm{inj}(E_\ell(t_{\ell-1}))\quad \text{(\Cref{lemma: g vector from rank vector})}.
    \end{equation}

    Every $i\in I$ is a source of $\Omega(t_i)$. Then we define a sequence of modules $M(t_{\ell-a})\in \rep H(t_{\ell-a})$ for $a = 1, \dots, \ell$ by the iteration
    \[
         M(t_{\ell-1}) \coloneqq E_{\ell}(t_{\ell-1}) \quad \text{and} \quad  M(t_{\ell-a-1}) = F_{\ell-a}^-(M(t_{\ell-a})).
    \]
    Using \Cref{prop: tau loc free reflection}, each $M(t_{\ell-a})$ is locally free rigid indecomposable and each reflection satisfies the assumption in \Cref{prop: recurrence under reflection uniform ver}. In particular, we have
    \[
        \rank M(t_{\ell-a}) = s_{\ell-a-1}\dots s_{\ell-1}(\alpha_\ell) \quad \text{and} \quad M(t_{0}) = P_{\ell}(t_0).
    \]
    Finally we apply the recurrence of \Cref{prop: recurrence under reflection uniform ver} to the path (from the right to the left)
    \[
        \begin{tikzcd}
            t_0 \ar[r, dash, "1"] & t_1 \ar[r, dash, "2"] & \cdots \ar[r, dash, "\ell-1"] & t_{\ell-1}
        \end{tikzcd}
    \]
    with the base case (\ref{eq: base case f poly g vec}). Notice that in the base case, we can define as in \Cref{def: h vector} $\mathbf h(M(t_{\ell-1}))$ because $F_{M(t_{\ell-1})}$ is in $\mathbb Q_\mathrm{sf}(y_1, \dots, y_n)$. Then by \Cref{lemma: h vector}, the $(\ell-1)$-th component of $\mathbf h(M(t_{\ell-1}))$ is exactly $h_{\ell-1}(M(t_\ell-1))$ of \Cref{prop: recurrence under reflection uniform ver}. Inductively using \Cref{prop: recurrence under reflection uniform ver} for $a=1, \dots, \ell$ and comparing with the recurrence \Cref{prop: recurrence g vector f poly dwz}, we have that each $F_{M(t_{\ell-a})}$ is in $\mathbb Q_\mathrm{sf}(y_1, \dots, y_n)$ (thus $\mathbf h(M(t_{\ell-a}))$ is well-defined), and
    \[
        F_{\ell; t_n}^{B_{t_{\ell-s}}; t_{\ell-a}} = F_{M(t_{\ell-a})} \quad \text{and} \quad
        \mathbf g_{\ell; t_n}^{B_{t_{\ell-a}}; t_{\ell-a}} = \mathbf g^\mathrm{inj}(M(t_{\ell-a})).
    \]
    In particular, letting $a = \ell$ finishes the proof.
\end{proof}

Let $\beta = c^r (\beta_\ell)$ for $\ell\in I$ and $r\in \mathbb N$. Replicate the mutation sequence (\ref{eq: sink sequence}) $r+1$ times to have the path in $\mathbb T_n$
\begin{equation}\label{eq: sink sequence r times}
    \begin{tikzcd}
        t_0 \ar[r, dash, "1"] & t_1 \ar[r, dash, "2"] & \cdots \ar[r, dash, "n-1"] & t_{(r+1)n-1} \ar[r, dash, "n"] & t_{(r+1)n}.
    \end{tikzcd}
\end{equation}

\begin{proposition}\label{prop: f poly g vec preprojective}
    Suppose that the Cartan matrix $C$ is not of finite type. For any $\ell$, $r\in \mathbb N$ and $\beta = c^r (\beta_\ell)$, we have
    \[
        \rank \tau^{-r}P_\ell(t_0) = \beta, \quad F_{\ell; t_{(r+1)n}}^{B; t_0} = F_{\tau^{-r}P_{\ell}(t_0)} \quad \text{and} \quad \mathbf g_{\ell; t_{(r+1)n}}^{B; t_0} = \mathbf g^{\mathrm{inj}}(\tau^{-r}P_\ell(t_0)).
    \]
\end{proposition}

\begin{proof}
    The sequence of matrices $B_{t_s}$ is periodic on the path (\ref{eq: sink sequence r times}). In particular, we have
    \[
        B = B_{t_0} = B_{t_n} = \cdots = B_{t_{(r+1)n}}.
    \]
    Apply \Cref{lemma: f poly g vec projective} to the sub-path
    \[
        \begin{tikzcd}
            t_{rn} \ar[r, dash, "1"] & t_{rn+1} \ar[r, dash, "2"] & \cdots \ar[r, dash, "n-1"] & t_{(r+1)n-1} \ar[r, dash, "n"] & t_{(r+1)n}.
        \end{tikzcd}
    \]
    Immediately we obtain
    \[
        \rank P_\ell(t_0) = \beta_\ell,\quad F_{\ell; t_{(r+1)n}}^{B; t_{rn}} = F_{P_{\ell}(t_0)} \quad \text{and} \quad \mathbf g_{\ell; t_{(r+1)n}}^{B; t_{rn}} = \mathbf g^{\mathrm{inj}}(P_\ell(t_0)).
    \]
    Now we define a sequence of modules related by (source) reflections. For $s = 0, \dots, r-1$ and $i = 1, \dots, n$, define
    \[
        M(t_{(r-s)n - i}) \coloneqq F_{n-i+1}^- \circ \cdots \circ F_{n}^- \circ (C^-)^s(P_\ell(t_0)) \in \rep H(t_{(r-s)n-i}). 
    \]
    This sequence of reflections satisfy the assumptions in \Cref{prop: recurrence under reflection uniform ver} in view of \Cref{prop: tau loc free reflection}. Therefore their rank vectors, $F$-polynomials, $\mathbf h$-vectors and $\mathbf g$-vectors satisfy the recurrence in \Cref{prop: recurrence under reflection uniform ver}. Again, compare with the exact same recurrence in \Cref{prop: recurrence g vector f poly dwz}. Hence inductively, we obtain in particular that for $s = 0, \dots, r$
    \[
        \rank M(t_{(r-s)n}) = c^s\beta_\ell,\quad F_{\ell; t_{(r+1)n}}^{B; t_{(r-s)n}} = F_{M(t_{(r-s)n})} \quad \text{and} \quad \mathbf g_{\ell; t_{(r+1)n}}^{B; t_{(r-s)n}} = \mathbf g^{\mathrm{inj}}(M(t_{(r-s)n})).
    \]
    Taking $s = r$ finishes the proof.
\end{proof}

Instead of the path (\ref{eq: sink sequence}), consider in $\mathbb T_n$ the following path
\begin{equation}\label{eq: source sequence}
    \begin{tikzcd}
        v_0 = t_0 \ar[r, dash, "n"] & v_1 \ar[r, dash, "n-1"] & \cdots \ar[r, dash, "2"] & v_{n-1} \ar[r, dash, "1"] & v_n
    \end{tikzcd}
\end{equation}
and replace (\ref{eq: sink sequence r times}) with
\begin{equation}\label{eq: source sequence r times}
    \begin{tikzcd}
        t_0 \ar[r, dash, "n"] & v_1 \ar[r, dash, "n-1"] & \cdots \ar[r, dash, "2"] & v_{(r+1)n-1} \ar[r, dash, "1"] & v_{(r+1)n}.
    \end{tikzcd}
\end{equation}

Let $\inj_\ell(t_0)$ denote the injective envelope of $E_\ell$ in $\rep H$. Following the same strategy in \Cref{lemma: f poly g vec projective} and \Cref{prop: f poly g vec preprojective}, we can prove

\begin{proposition}\label{prop: f poly g vec preinjective}
    Suppose that the Cartan matrix $C$ is not of finite type. For any $\ell$, $r\in \mathbb N$ and $\beta = c^{-r} (\gamma_\ell)$, we have
    \[
        \rank \tau^{r}\inj_\ell(t_0) = \beta, \quad F_{\ell; v_{(r+1)n}}^{B; t_0} = F_{\tau^{r}\inj_{\ell}(t_0)} \quad \text{and} \quad \mathbf g_{\ell; v_{(r+1)n}}^{B; t_0} = \mathbf g^{\mathrm{inj}}(\tau^{r}\inj_\ell(t_0)).
    \]
\end{proposition}

In view of the classification of real Schur roots in affine case (see \Cref{prop: real schur roots affine type}), we are left with $\beta\in \Delta_\mathrm{rS}$ with finite $c$-orbit. By \Cref{thm: real schur roots rigid module}, there is a unique indecomposable locally free rigid $H$-module $M(\beta)$ such that $\rank M(\beta) = \beta$. There is also a cluster variable $X_\beta\in \mathcal A_\bullet(B)$ such that $\mathbf d(X_\beta) = \beta$. Denote by $F(X_\beta)$ and $\mathbf g(X_\beta)$ respectively the $F$-polynomial and $\mathbf g$-vector of $X_\beta$ with respect to the initial seed $(B, (x_1, \dots, x_n, y_1, \dots, y_n))$.

\begin{proposition}\label{prop: f poly g vec finite c orb}
    If $\beta\in \Delta_\mathrm{rS}$ has a finite $c$-orbit, then $F(X_\beta) = F_{M(\beta)}$ and $\mathbf g(X_\beta) = \mathbf g^\mathrm{inj}(M(\beta))$.
\end{proposition}

\begin{proof}
    The crucial point here is that in any finite $c$-orbit there always exists at least one root in the finite root system $\Phi_\mathrm{fin}$, as already discussed in \Cref{section: coxeter action}. As in \Cref{subsection: coxeter orbit}, we choose an extended vertex $k$ and let $I_\mathrm{fin} = I \setminus \{k\}$. Consider the $I_\mathrm{fin} \times I_\mathrm{fin}$ submatrix $C_\mathrm{fin}$ of $C$, with the induced symmetrizer $D_\mathrm{fin}$ and orientation $\Omega_\mathrm{fin}$. Then there is the associated GLS algebra $H_\mathrm{fin} = H(C_\mathrm{fin}, D_\mathrm{fin}, \Omega_\mathrm{fin})$, naturally a subalgebra of $H$. Denote $B_\mathrm{fin} = B(C_\mathrm{fin}, \Omega_\mathrm{fin})$.

    Suppose that $\beta$ is a positive root in $\Phi_\mathrm{fin} \cap U^c$. Then there is a unique rigid locally free $H_\mathrm{fin}$-module $M(\beta)$ such that $\rank M(\beta) = \beta$ by \cite{MR3660306}. Note that $M(\beta)$ can be viewed as an $H$-module, that is still locally free and rigid. Our first step is to show that the proposition holds for such $\beta$.

    Consider the finite type cluster algebra $\mathcal A_\bullet (B_\mathrm{fin})$. Denote by $X_\beta^\mathrm{fin}$ the unique cluster variable whose $\mathbf d$-vector is $\beta$. Notice that we have
    \begin{equation}\label{eq: f-poly finite}
        F_{M(\beta)}^\mathrm{fin}(y_i\mid i\in I_\mathrm{fin}) = F({X_\beta^\mathrm{fin}})(y_i\mid i\in I_\mathrm{fin}).
    \end{equation}
    This is a theorem of Geiss--Leclerc--Schr\"oer \cite{MR3830892}. Alternatively it can be proven in the same style as \Cref{lemma: f poly g vec projective} and \Cref{prop: f poly g vec preprojective}.

    Rupel proved in \cite{MR2817677} a module theoretic formula $F(X_\beta^\mathrm{fin}) = F_{\beta}$; see \cite[Corollary 3]{MR4057626}. Here we denote his module theoretic polynomial by $F_{\beta}$. Then by (\ref{eq: f-poly finite}), we have $F_{\beta} = F_{M(\beta)}^\mathrm{fin}$. Rupel's formula was extended beyond finite types in \cite{MR3378824}. In particular, let $Q$ be the affine type valued quiver associated to $B$ on the vertex set $I$ containing $Q_\mathrm{fin}$ on $I_\mathrm{fin}$ as a subquiver. Manifestly from his methods, since a rigid representation $V$ of the valued quiver $Q_\mathrm{fin}$ over a finite field $\mathbb F$ can be viewed as a rigid representation of $Q$, the polynomial $F_\beta$ also serves as the $F$-polynomial of $X_\beta$ in $\mathcal A_\bullet(B)$. Notice that in our setup $F_{M(\beta)} = F_{M(\beta)}^\mathrm{fin}$ by definition, hence $F_{M(\beta)} = F(X_\beta)$.

    The cluster character formula in \cite{MR3378824} (evaluating the quantum parameter $q$ at 1) can be written as
    \[
        X_\beta = \prod_{i\in I} x_i^{-v_i + \sum_{j\in I}[-b_{ij}]v_j} F_{\beta}(\hat y\mid i\in I),
    \]
    where $(v_i)_{i\in I} = \beta$. Thus 
    \[
        \mathbf g(X_\beta) = \deg X_\beta = (-v_i + \sum_{j\in I}[-b_{ij}]v_j)_{i\in I},
    \]
    which by \Cref{lemma: g vector from rank vector} equals $\mathbf g^\mathrm{inj}(M(\beta))$.

    Now we consider the real Schur root $c\beta$. Suppose that $X_\beta$ is realized as a cluster variable $x_{\ell; t}^{B; t_0}$ at some $t\in \mathbb T_n$. Consider the automorphism of (edge-labeled) $\mathbb T_n$ induced by sending $t_0$ to $t_n$ in (\ref{eq: sink sequence}). Suppose that this automorphism sends $t$ to $t'$. We have proven that
    \[
        F_{\ell; t'}^{B; t_n} = F_{M(\beta)} \quad \text{and} \quad \mathbf g_{\ell; t'}^{B; t_n} = \mathbf g^\mathrm{inj}(M(\beta)).
    \]
    Apply the sequence of source reflections $F_n^-$, $F_{n-1}^-$, \dots, $F_1^-$ (from left to right) to $M(\beta)$. Along the way, each reflection satisfies the assumptions in \Cref{prop: recurrence under reflection uniform ver}. In particular, $M(c\beta) = \tau^{-1} M(\beta)$. This sequence of reflections moves the base vertex from $t_n$ to $t_0$. Therefore, as in the infinite $c$-orbit case, by comparing with the same recurrence in \Cref{prop: recurrence g vector f poly dwz}, we obtain
    \[
        F(X_{c\beta}) = F_{\ell; t'}^{B; t_0} = F_{M(c\beta)} \quad \text{and} \quad \mathbf g(X_{c\beta}) = \mathbf g_{\ell; t'}^{B; t_0} = \mathbf g^\mathrm{inj}(M(c\beta)).
    \]

    Repeating the above process proves the result for all roots in the $c$-orbit of $\beta$. Since there is always one root in $\Phi_\mathrm{fin}$ of a finite $c$-orbit of real Schur roots, the proof is complete.
\end{proof}

Finally we summarize the main result in this section.

\begin{theorem}\label{thm: main bijection cc function}
    Let $C$ be of affine type. Sending a locally free $H_\mathbb C(C, D, \Omega)$-module $M$ to its Caldero--Chapoton function $X_M^\bullet$ induces a bijection from indecomposable rigid locally free modules (up to isomorphism) to non-initial cluster variables
    \[
        \left\{ M(\beta)\in \rep H \mid \beta \in \Delta_{\mathrm{rS}} \right\} \quad \xlongrightarrow{\sim} \quad
        \left\{ X_\beta \in \mathcal A_\bullet(B) \mid \beta \in \Delta_{\mathrm{rS}} \right\}.
    \]
\end{theorem}

\begin{proof}
    The modules in the statement are parametrized by their rank vectors as real Schur roots. Thus they are either of type $\tau^{-r} P_\ell$, or $\tau^r \inj_\ell$ with $r\in \mathbb Z_{\geq 0}$, or in a finite $\tau$-orbit. 
    
    We have shown in \Cref{prop: f poly g vec preprojective}, \Cref{prop: f poly g vec preinjective} and \Cref{prop: f poly g vec finite c orb} that for a module in any of these three cases, its $F$-polynomial and injective $\mathbf g$-vector equal the $F$-polynomial and $\mathbf g$-vector of some cluster variable. It is clear from the CC formula $X_M^\bullet$ that the cluster variable $X^\bullet_{M(\beta)}$ has $\mathbf d$-vector $\beta$ as the rank vector of the module $M(\beta)$. By \Cref{thm: d vectors real schur roots}, these cluster variables $X^\bullet_{M(\beta)}$ for $\beta\in \Delta_{\mathrm{rS}}$ are all the non-initial cluster variables in $\mathcal A_\bullet(B)$. This concludes the bijection.
\end{proof}

\begin{remark}
    In Rupel's formula on $F$-polynomials, the coefficients are evaluations (at $q = 1$) of certain polynomials counting submodules in rigid representations over finite fields. The existence of these counting polynomials is a result from that the corresponding CC functions for any finite field $\mathbb F$ are indeed evaluations of quantum cluster variables at $q = |\mathbb F|$. Our formula on the other hand utilizes Euler charactersitics of (locally free) quiver Grassmannians over complex numbers, thus providing a direct and explicit expression for cluster variables. The seemingly different formulae $F_{M(\beta)}$ and $F_\beta$ from the definitions are actually equal.
\end{remark}

\begin{remark}
    Evaluating every $y_i$ at $1$, one obtains the corresponding result for the coefficient-free cluster algebra $\mathcal A(B)$ as in the form of \Cref{thm: first main theorem intro}. In this case the Laurent polynomials $X_{M(\beta)}$ are still parametrized by their $\mathbf d$-vectors.
\end{remark}

\begin{remark}\label{rmk: cc function of cluster monomial}
    According to \cite{MR4125687}, cluster monomials are in bijection with pairs $(M, \inj)$ such that $M\in \replf H$ is rigid and $\inj$ is injective with $\Hom_H(M, \inj) = 0$. These are $\tau^-$-rigid pairs of Adachi--Iyama--Reiten \cite{MR3187626}. We set $(0, \inj_i)$ to correspond to the initial cluster variable $x_i$. The indecomposable summands of such $(M, \inj)$, which are either $(N, 0)$ with $N$ being indecomposable or $(0, \inj_i)$, correspond to cluster variables belonging to one cluster. Writing $M = \bigoplus_{s\in S} M_s^{b_s}$ where $M_s$ is indecomposable and $\inj = \bigoplus_{i=1}^n \inj_i^{a_i}$, we can define
    \[
        X^\bullet_{(M, \inj)} \coloneqq X^\bullet_M \prod_{i=1}^n x_i^{a_i} \quad \overset{\text{\Cref{rmk: cc function direct sum}}}{=\joinrel=} \quad \prod_{s\in S} (X^\bullet_{M_s})^{b_s}\prod_{i=1}^n x_i^{a_i},
    \]
    which provides a formula of the corresponding cluster monomial.
\end{remark}

\section{Generic bases}\label{section: generic bases}

\subsection{Affine spaces of locally free modules}\label{subsec: affine spaces loc free}

With fixed $\mathbf r\in \mathbb N^n$, consider the free $H_i$-modules $V_i = H_i^{\oplus r_i}$ for $i\in I$. Define
\[
    \replf(H, \mathbf r) \coloneqq \bigoplus_{(i,j)\in \Omega} \Hom_{H_i}({_iH_j}\otimes_{H_j}V_j, V_i) = \prod_{(i,j)\in \Omega} \mathrm{Mat}_{r_i\times b_{ij}r_j}(H_i).
\]
It is an affine space over the ground field $K$ and can be viewed as the space of rank $\mathbf r$ locally free representations of the modulated graph $({_iH_j}, H_i)$. The (algebraic) group
\[
    \mathrm{GL}_\mathbf{r}(H) \coloneqq \prod_{i\in I} \mathrm{GL}_{r_i}(H_i) = \prod_{i\in I} \mathrm{Aut}_{H_i}(V_i)
\]
acts on $\replf(H, \mathbf r)$ by
\[
    (g_i)_{i\in I} \cdot (M_{ij})_{(i,j)\in \Omega} = \left( g_i\circ M_{ij}\circ (\mathrm{id}_{{_iH_j}}\otimes g_j^{-1}) \right)_{(i,j) \in \Omega}.
\]
The orbits are naturally in bijection with isomorphism classes of representations. We note that $\replf(H, \mathbf r)$ is denoted differently as $\mathrm{rep}^\mathrm{fib}_\mathrm{l.f.}(H, \mathbf f)$ in \cite[Section 3.1]{MR3801499}.

As discussed earlier in \Cref{subsec: reps of H}, the space $\Hom_{H_i}({_iH_j}\otimes_{H_j}V_j, V_i)$ can be ($\mathrm{GL}_\mathbf{r}(H)$-equivariently) identified with
\[
    \Hom_{H_j}(V_j, {_jH_i}\otimes_{H_i} V_i) = \mathrm{Mat}_{r_j\times (-b_{ji}r_i)}(H_j).
\]

Recall that locally free rigid $H$-modules are parametrized by their rank vectors \Cref{section: reps of h}. If $\mathbf r$ is such a rank vector, by standard arguments there is a (unique) open $\mathrm{GL}_\mathbf r(H)$-orbit in $\replf(H, \mathbf r)$ of rigid modules; see for example the discussion in \cite[Section 3.4]{pfeifer2023generic}.

\subsection{\texorpdfstring{Generic $F$-polynomials}{Generic F-polynomials}}\label{subsec: generic f poly}

Let $\gr^{H_i}_e(r)$ denote the quasi-projective variety of free $H_i$-submodules of rank $e$ in $H_i^{\oplus r}$. Recall $S = \prod_{i\in I} H_i$ from \Cref{subsec: loc free module}. For rank vectors $\mathbf e = (e_i)_{i\in I}\in \mathbb N^n$ and $\mathbf r = (r_i)_{i\in I}\in \mathbb N^n$, we define
\[
    \gr^{S}_\mathbf e(\mathbf r) = \prod_{i\in I} \gr^{H_i}_{e_i}(r_i).
\]
It can be viewed as the set of locally free $S$-submodules of rank vector $\mathbf e$ in $\bigoplus_{i\in I}H_i^{r_i}$.

Next we define a space over $\replf(H, \mathbf r)$ whose fiber over a point $M\in \replf(H, \mathbf r)$ is $\gr^H_\mathbf e(M)$. Consider the \emph{incidence} space
\begin{equation}
    \gr^H_\mathbf e (\mathbf r) \coloneqq \{(N, M)\in \gr^S_\mathbf e(\mathbf r) \times \replf(H, \mathbf r) \mid M_{ij}({_iH_j}\otimes_{H_j}N_j) \subseteq N_i \ \text{for}\ (i,j)\in \Omega \}.
\end{equation}
It comes with two forgetful projections 
\[
    \begin{tikzcd}
         & \gr^H_\mathbf e(\mathbf r) \ar[ld, swap, "p_1"] \ar[rd, "p_2"] & \\
        \gr_\mathbf e^S(\mathbf r) & & \replf(H, \mathbf r).
    \end{tikzcd}
\]
The map $p_1$ is clearly a vector bundle. The fiber of $p_2$ at $M\in \replf(H, \mathbf r)$ is just $\gr_\mathbf e^H(M)$.

\begin{lemma}\label{lemma: constructible cc map}
    For any rank vector $\mathbf e$, there is a non-empty dense open subset $U_{\mathbf e}$ of $\replf(H, \mathbf r)$ such that the map
    \[
        f\colon M \mapsto \chi(\gr^H_\mathbf e(M))
    \]
    is constant on $U_{\mathbf e}$.
\end{lemma}

\begin{proof}
    Consider the map (between complex algebraic varieties)
    \[
        p_2 \colon \gr_\mathbf e^H(\mathbf r) \rightarrow \replf(H, \mathbf r).
    \]
    Then the function $f$ is the pushforward of the constant function on $\gr_\mathbf e^H(\mathbf r)$, that is,
    \[
        f(M) = \chi(p_2^{-1}(M)), \quad M\in \replf(H, \mathbf r).
    \]
    It follows from a general theory of morphisms between algebraic varieties \cite{MR0361141} that $f$ is constructible on $\replf(H, \mathbf r)$. Then there exists a finite decomposition of $\replf(H, \mathbf r)$ into constructible sets each with constant function value. Since $\replf(H, \mathbf r)$ is an affine space (thus irreducible), there is an open subset $U_{\mathbf e}$ in the decomposition which is dense.
\end{proof}

The above lemma ensures that there is a generic $F$-polynomial defined for each $\replf(H, \mathbf r)$.

\begin{corollary}\label{cor: open subset f poly}
    There is a non-empty open subset $U \subset \replf(H, \mathbf r)$ such that the assignment
    \[
        M \mapsto F_M, \quad M\in \replf(H, \mathbf r)
    \]
    is constant on $U$.
\end{corollary}

\begin{proof}
    Let $U$ be the finite intersection $\bigcap_{\mathbf e \leq \mathbf r} U_{\mathbf e}$. It is (non-empty) open in $\replf(H, \mathbf r)$ because $\replf(H, \mathbf r)$ is irreducible. Then for any $\mathbf e$, the map $M \mapsto \chi(\gr^H_\mathbf e(M))$ is constant on $U$ by \Cref{lemma: constructible cc map}, thus so is the map $M \mapsto F_M$.
\end{proof}

\begin{definition}\label{def: generic f poly cc function}
    For a rank vector $\mathbf r\in \mathbb N^n$, we define the \emph{generic $F$-polynomial} $F^H_{\mathbf r} = F_{\mathbf r}$ to be $F_M$ for (any) $M\in U\subset \replf(H, \mathbf r)$ (as in \Cref{cor: open subset f poly}).
\end{definition}

\subsection{Decorations}\label{subsec: decoration}

For the purpose of relating $H$-modules with cluster algebras, following \cite{MR2480710} we introduce a convenient notion called \emph{decorated representations}.

For any $\mathbf v = (v_i)_{i\in I}\in \mathbb Z^n$, let
\[
    \mathbf v^+ = (v^+_i)_i \coloneqq ([v_i]_+)_i \in \mathbb N^n \quad \text{and} \quad \mathbf v^- = (v^-_i)_i \coloneqq ([-v_i]_+)_i \in \mathbb N^n.
\]
The two vectors $\mathbf v^+$ and $\mathbf v^-$ have disjoint supports, i.e., $v^+_iv^-_i = 0$ for any $i$. We define the notation
\[
    \dreplf(H, \mathbf v) \coloneqq (\replf(H, \mathbf v^+), \mathbf v^-)
\]
where $\mathbf v^-$ is viewed as a \emph{decoration}. A point $\mathcal M = (M, \mathbf v^-)$ in $\dreplf(H, \mathbf v)$ can be regarded as a \emph{decorated representation} with $M\in \replf H$ of rank $\mathbf v^+$ and decoration $\mathbf v^-$. Operations of representations such as direct sum naturally extend to decorated representations.

In this paper, we always want the decoration $\mathbf v^-$ to be disjoint from the support of $M$, which is less general than \cite{MR2480710}. The definition of generic $F$-polynomial is extended to any $\mathbf v\in \mathbb Z^n$ by setting $F^H_{\mathbf v} = F_{\mathbf v} \coloneqq F_{\mathbf v^+}$.

\subsection{\texorpdfstring{Reflections of generic $F$-polynomials}{Reflections of generic F-polynomials}}

Let $k\in I$ be a sink of $\Omega$. Write $\mathbf v = \mathbf v^+ - \mathbf v^-$ as in \Cref{subsec: decoration}. Define the vector $\mathbf v' = (v'_i)_i \in \mathbb Z^n$ by
\begin{equation}\label{eq: reflection decorated rank vector}
    v_i' \coloneqq \begin{dcases}
        v_i \quad & \text{if $i\neq k$} \\
        -v_k + \sum_{j = 1}^n [b_{kj}]_+v^+_j \quad & \text{if $i = k$}.
    \end{dcases}
\end{equation}
Sending $\mathbf v$ to $\mathbf v'$ defines a bijection from $\mathbb Z^n$ to $\mathbb Z^n$. We view $\mathbf v'$ as a (decorated) rank vector for the algebra $H' = H(C, D, s_k(\Omega))$.

\begin{proposition}\label{prop: reflection generic f poly}
    Let $k$ be a sink. For any $\mathbf v\in \mathbb Z^n$ and $\mathbf v'$ as defined in (\ref{eq: reflection decorated rank vector}), the generic $F$-polynomials satisfy the equation
    \begin{equation}\label{eq: reflection generic value}
        (1 + y_k)^{-v_k^+}F^H_{\mathbf v}(y_1, \dots, y_n) = (1 + y_k')^{-v_k^-} F^{H'}_{\mathbf v'}(y_1',\dots, y_n'),
    \end{equation}
    where $y_i' = y_iy_k^{[b_{ki}]_+}(y_k+1)^{-b_{ki}}$ for $i\neq k$ and $y_k' = y_k^{-1}$.
\end{proposition}

\begin{proof}

    For $i\in I$ a sink (resp. source), let $\drepfr(H, \mathbf v, i)$ denote the open subset of $\dreplf(H, \mathbf v)$ with full rank $M_{i, \mathrm{in}}$ (resp. $M_{i, \mathrm{out}}$). The group $\mathrm{GL}_{v_i^+}(H_i)$ acts on $\drepfr(H, \mathbf v, i)$ by change of $H_i$-bases in $H_i^{v_i^+}$.

    Recall that we have defined $\gr_{e}^{H_k}(r)$ the Grassmannian of free $H_k$-submodules of rank $e$ in $H_k^r$ for $e\leq r$. Dually we denote $\mathrm{Quot}_e^{H_k}(r)$ the set of free factor modules of $H_k^r$ of rank $e$. There is a natural isomorphism
    \begin{equation}\label{eq: quot to grass}
        \varphi \colon \mathrm{Quot}_{r-e}^{H_k}(r) \rightarrow \gr_e^{H_k}(r), \quad V \mapsto \ker(H_k^r \twoheadrightarrow V).
    \end{equation}

    Define
    \[
        v_{k, \mathrm{in}} \coloneqq \sum_{(k, j)\in \Omega} b_{kj}v_j^+ \quad \text{and} \quad
        v'_{k, \mathrm{out}} \coloneqq \sum_{(j, k)\in \Omega'} -b'_{kj}(v'_j)^+.
    \]
    Note that these two numbers equal.
    
    Assume first that $v_k \leq 0$, i.e., $v_k^+ = 0$. In this case $\drepfr(H, \mathbf v, k) = \dreplf(H, \mathbf v)$. A decorated module $M$ belongs to $\drepfr(H', \mathbf v', k)$ if and only if 
    \[
        M_{k, \mathrm{out}} \colon H_k^{v_k^- + v_{k, \mathrm{out}}'} \rightarrow H_k^{v_{k, \mathrm{out}}'}
    \]
    is of full rank, i.e., surjective. Then it must be isomorphic to $\mathcal M \oplus E_k^{v_k^-}$ where $\mathcal M$ is in $\drepfr(H', \tilde {\mathbf v}', k)$ and $\tilde {\mathbf v}' = \mathbf v' - v_k^-\alpha_k$. Notice that $G = \mathrm{GL}_{v_{k, \mathrm{out}}'}(H_k)$ acts freely on $\drepfr(H', \tilde {\mathbf v}', k)$ fixing isomorphism classes. The quotient $\drepfr(H', \tilde{\mathbf v}', k)/G$ is identified with the affine space $\dreplf(H', \bar{\mathbf v}')$, where $\bar{\mathbf v}' = \tilde{\mathbf v}' - v_{k,\mathrm{out}}'\alpha_k$ (thus with zero $k$-component).

    Then there is a projection $\drepfr(H', \mathbf v', k)$ to $\drepfr(H', \tilde{\mathbf v}', k)/G$ by forgetting $M_{k, \mathrm{out}}$ that sends (the isoclass of) $\mathcal M\oplus E_k^{v_k^-}$ to $\mathcal M$. Take an open subset $U$ in $\drepfr(H', \tilde {\mathbf v}', k)/G$ where the $F$-polynomial takes the generic value $F_{\tilde {\mathbf v}'}$. Pull $U$ back to $\drepfr(H', \mathbf v', k)$ and we know that the $F$-polynomials there are that of $\mathcal M\oplus E_k^{v_k^-}$ for $\mathcal M\in U$. Therefore we have (by \Cref{rmk: cc function direct sum})
    \[
        F^{H'}_{\mathbf v'}(y_1',\dots, y_n') = (1 + y_k')^{v_k^-}F^{H'}_{\tilde {\mathbf v}'}(y_1', \dots, y_n').
    \]
    
    Then notice that the reflection functor $F_k^+$ precisely gives an isomorphism 
    \[
        f \colon \drepfr(H, \mathbf v, k) \xlongrightarrow{\sim} \drepfr(H', \tilde {\mathbf v}', k)/G.
    \]
    The generic $F$-polynomial $F^H_\mathbf v$ can then be taken within (an open subset of) $f^{-1}(U)$. In particular one can find $\mathcal M\in f^{-1}(U)$ such that $F^H_\mathbf v = F^H_\mathcal M$ and $F_{\tilde {\mathbf v}'}^{H'} = F^{H'}_{\mathcal M'}$ for $\mathcal M'\in U$ isomorphic to $F_k^+(\mathcal M)$. Now (\ref{eq: reflection generic value}) follows from \Cref{prop: recurrence under reflection uniform ver} (which relates $F^H_\mathbf v$ with $F^{H'}_{\tilde {\mathbf v}'}$) in this case.

    Next we consider the case where $v_k \geq 0$, i.e., $v_k^- = 0$. If $v_{k, \mathrm{in}}\leq v_k$, then the situation is similar to the previous case. Now a decorated module in $\drepfr(H, \mathbf v, k)$ is isomorphic to $\mathcal M \oplus E_k^{v_k - v_{k, \mathrm{in}}}$ where $\mathcal M$ is in $\drepfr(H, \mathbf v - v_{k, \mathrm{in}}\alpha_k, k)$. The result eventually follows from \Cref{prop: recurrence under reflection uniform ver} using a similar argument as in the previous case.

    Finally consider the situation when $v_{k, \mathrm{in}}\geq v_k$. Now $v'_k = v_{k, \mathrm{in}} - v_k \geq 0$. Let $\mathbf v_{\hat k} \coloneqq \mathbf v - v_k\alpha_k$ and $\mathbf v'_{\hat k} \coloneqq \mathbf v' - v_k'\alpha_k$. Then we have
    \begin{align*}
        & \drepfr(H, \mathbf v, k) = \dreplf(H, \mathbf v_{\hat k}) \times \mathrm{Mat}^\circ_{v_k \times v_{k, \mathrm{in}}}(H_k), \\
        & \drepfr(H', \mathbf v', k) = \dreplf(H', \mathbf v_{\hat k}') \times \mathrm{Mat}^\circ_{v_{k, \mathrm{out}}'\times v_k'}(H_k),
    \end{align*}
    where $\mathrm{Mat}^\circ$ means full rank matrices. Notice that $\dreplf(H, \mathbf v_{\hat k}) = \dreplf(H', \mathbf v_{\hat k}')$. 
    There is the following diagram:
    \[
        \begin{tikzcd}
            \drepfr(H, \mathbf v, k) \ar[d, swap, "\pi"] & \drepfr(H', \mathbf v', k) \ar[d, "\pi'"]\\
            \dreplf(H, \mathbf v_{\hat k}) \times \mathrm{Quot}^{H_k}_{v_k}(H_k^{v_{k,\mathrm{in}}}) \ar[r, "\sim", "{(\mathrm{id}, \varphi)}"'] & \dreplf(H', \mathbf v'_{\hat k}) \times \gr^{H_k}_{v_k'}(H_k^{v_{k,\mathrm{out}}'})
        \end{tikzcd}
    \]
    The map $\pi$ (resp. $\pi'$) is taking the quotient by the free action of $\mathrm{GL}_{v_k}(H_k)$ (resp. $\mathrm{GL}_{v_k'}(H_k)$) on $\mathrm{Mat}^\circ_{v_k\times v_{k, \mathrm{in}}}(H_k)$ (resp. $\mathrm{Mat}^\circ_{v_{k, \mathrm{out}}'\times v_k'}(H_k)$). Thus these vertical maps are principal $G$-bundles. The isomorphism in the bottom row is induced precisely by the reflection functor $F_k^+$, where $\varphi$ is described in (\ref{eq: quot to grass}). Then we can take open subsets $U$ and $U'$ respectively of the domain and codomain such that $U'= (\mathrm{id}, \varphi)(U)$ and generic $F$-polynomials are respectively taken within $U$ and $U'$. This means $F^{H}_\mathbf v = F^H_{\mathcal M}$ for any $\mathcal M\in \pi^{-1}(U)$ and $F^{H'}_{\mathbf v'} = F^{H'}_{\mathcal M'}$ for any $\mathcal M'\in (\pi')^{-1}(U')$. In particular, one can choose $\mathcal M' \cong F_k^+(\mathcal M)$. Then by \Cref{prop: recurrence under reflection uniform ver}, the generic $F$-polynomials $F_\mathbf v^H$ and $F_{\mathbf v'}^{H'}$ are related exactly as (\ref{eq: reflection generic value}).
\end{proof}

We remark that \Cref{prop: reflection generic f poly} is valid for general $C$ not only restricted to affine or finite types.

\subsection{\texorpdfstring{Correspondence between injective $\mathbf g$-vectors and rank vectors}{Correspondence between g-vectors and d-vectors}}\label{subsec: between g and d vectors}

We define a bijection $\mathbb{Z}^n \to \mathbb{Z}^n$ that is expected to relate injective $\mathbf g$-vectors to rank vectors of decorated representations. This construction can be viewed as the representation-theoretic analogue of the familiar correspondence between $\mathbf g$- and $\mathbf d$-vectors for acyclic cluster algebras; see \cite{BMR,FuKeller} for the skew-symmetric case.

\begin{definition}
    For $\mathcal M = (M, \mathbf v^-) \in \dreplf(H, \mathbf v)$, we define the \emph{injective $\mathbf g$-vector} by
    \[
        \mathbf g^\mathrm{inj}_H(\mathcal M) \coloneqq \mathbf g^\mathrm{inj}_H(M) + \mathbf v^-.
    \]
\end{definition}

Let $\mathbf v = (v_i)_{i\in I}$ in $\mathbb Z^n$ viewed as a decorated rank vector. The corresponding $\mathbf g$-vector $\mathbf g = (g_i)_{i\in I}$ is given by
\[
    g_i = v_i^- - v_i^+ + \sum_{j\in I} [-b_{ij}]_+v_j^+ = -v_i + \sum_{j\in I} [-b_{ij}]_+v_j^+.
\]
It is $\mathbf g^\mathrm{inj}(\mathcal M)$ for any $\mathcal M\in \dreplf(H, \mathbf v)$. Sending $\mathbf v$ to $\mathbf g$ gives a bijection from $\mathbb Z^n$ to $\mathbb Z^n$.

Now we describe the inverse of the above map. Start with a sink $i_1$ of $\Omega$. Suppose that $i_{t}$ is a sink with $i_1, \dots, i_{t-1}$ removed. Define
\[
    v_{i_1} \coloneqq -g_{i_1} \quad \text{and} \quad v_{i_t} \coloneqq -g_{i_t} + \sum_{s=1}^{t-1}[-b_{i_t, i_s}]_+v_{i_s}^+ = -g_{i_t} + \sum_{s=1}^{t-1}-b_{i_t, i_s}v_{i_s}^+
\]
inductively. In this way we recover $\mathbf v$ from $\mathbf g$.

\subsection{Reflections of generic CC functions with coefficients}

Suppose that $\widetilde B = (\tilde b_{ij})\in \mathrm{Mat}_{m\times n}(\mathbb Z)$ is extended from $B$. Let $\widetilde B' = (\tilde b_{ij}') = \mu_k(\widetilde B)$.

Define a piecewise linear bijection $T_k\colon \mathbb Z^m \rightarrow \mathbb Z^m$, sending $\tilde{\mathbf g} = (g_i)_{i=1}^m$ to $\tilde {\mathbf g}' = (g'_i)_{i=1}^m$, by
\begin{equation}\label{eq: mutation rule extended g vector}
    g_k' = - g_k \quad \text{and} \quad g_i' = \begin{dcases}
            g_i + [-b_{ik}]_+g_k \quad & \text{if $g_k\leq 0$} \\
            g_i + [b_{ik}]_+g_k \quad & \text{if $g_k\geq 0$}
        \end{dcases} \quad \text{for $i\neq k$}.
\end{equation}

Following \Cref{subsec: between g and d vectors}, let $\mathbf v$ (resp. $\mathbf v'$) correspond to $\mathbf g = (g_i)_{i=1}^n$ (resp. $\mathbf g' = (g'_i)_{i=1}^n$), the \emph{principal part} of $\tilde {\mathbf g}$ (resp. $\tilde {\mathbf g}'$). The following lemma is straightforward.

\begin{lemma}\label{lemma: transformation principal rank vector}
    The vectors $\mathbf v$ and $\mathbf v'$ are related exactly by (\ref{eq: reflection decorated rank vector}).
\end{lemma}

\begin{definition}\label{def: generic cc function coef}
    For each $\tilde {\mathbf g}\in \mathbb Z^m$, the generic CC function with coefficients is
    \begin{equation}\label{eq: generic cc function coefficients}
        X^{\widetilde B}_{\tilde {\mathbf g}} (x_1, \dots, x_m) \coloneqq x^{\tilde {\mathbf g}} \cdot F_{\mathbf v}(\hat y_1, \dots, \hat y_n) \in \mathbb Z[x_1^\pm, \dots, x_m^\pm]
    \end{equation}
    where $\hat y_i = \prod_{j=1}^m x_j^{\tilde b_{ji}}$. 
\end{definition}

Next we present the reflection of generic CC function with coefficients.

\begin{proposition}\label{prop: reflection generic cc function coef}
    Let $k$ be a sink or source. Then we have
    \[
        X_{\tilde {\mathbf g}}^{\widetilde B}(x_1, \dots, x_m) = X_{\tilde {\mathbf g}'}^{\widetilde B'}(x_1', \dots, x_m')
    \]
    where $x_i' = x_i$ for $i\neq k$ and $x_k' = x_k^{-1}\left(\prod_{j=1}^m x_j^{[\tilde b_{jk}]_+} + \prod_{j=1}^m x_j^{[-\tilde b_{jk}]_+}\right)$.
\end{proposition}

\begin{proof}
    This proposition is a corollary of \Cref{prop: reflection generic f poly}. We first observe that for each $i\in I$, we have
    \[
        \prod_{j=1}^m (x_j')^{\tilde b_{ji}'} = y_i'\mid (\hat y_i \leftarrow y_i, \hat y_k \leftarrow y_k)
    \]
    where $y_i' = y_iy_k^{[b_{ki}]_+}(y_k+1)^{-b_{ki}}$ for $i\neq k$ and $y_k' = y_k^{-1}$.

    Then by \Cref{prop: reflection generic f poly}, it amounts to showing (when $k$ is a sink)
    \[
        x^{\tilde {\mathbf g}} \left(1 + \prod_{j=1}^m x_j^{\tilde b_{jk}}\right)^{v_k^+} = (x')^{\tilde {\mathbf g}'} \left(1 + \prod_{j=1}^m (x'_j)^{\tilde b'_{jk}}\right)^{v_k^-}.
    \]
    This is checked by using the explicit formula (\ref{eq: mutation rule extended g vector}) relating $\tilde {\mathbf g}$ and $\tilde {\mathbf g}'$.
\end{proof}

\subsection{Generic bases}

This section is devoted to prove

\begin{theorem}\label{thm: generic basis}
    For any $\widetilde B$ (extended from $B$ of affine type) of full rank, the following set $\mathcal S$ of generic CC functions
    \[
        \mathcal S = \{ X_{\tilde {\mathbf g}}^{\widetilde B} \mid {\tilde {\mathbf g}}\in \mathbb Z^m \}\subseteq \mathbb Z[x_1^\pm,\dots, x_{m}^\pm]
    \]
    is a $\mathbb Z$-basis of the cluster algebra $\mathcal A(\widetilde B)$ containing all cluster monomials.
\end{theorem}

We start with the following lemma.

\begin{lemma}\label{lemma: generic cc function univ laurent}
    Any generic CC function $X_{\tilde{\mathbf g}}^{\widetilde B}$ for $\tilde {\mathbf g}\in \mathbb Z^m$ is an element in $\mathcal A(\widetilde B)$.
\end{lemma}

\begin{proof}
    Consider a sink/source mutation sequence from $t_0$ to $t$ in the tree $\mathbb T_n$. This induces a bijection
    \[
        T_{t_0, t} \colon \mathbb Z^m \rightarrow \mathbb Z^m
    \]
    by composing the maps $T_k$ (see (\ref{eq: mutation rule extended g vector})) along the mutation sequence. Denote $\tilde {\mathbf g}_t = T_{t_0, t}(\tilde {\mathbf g})$. By \Cref{prop: reflection generic cc function coef}, equivalently it suffices to show that
    \[
        X_{\tilde{\mathbf g}_t}^{\widetilde B_{t}}(x_{1;t}, \dots, x_{m; t})\in \mathbb Z[x_{1;t}^\pm, \dots, x_{m; t}^\pm]
    \]
    belongs to ${\mathcal A}(\widetilde B_t)$. By \cite[Proposition 1.8]{MR2110627}, for any $t\in \mathbb T_n$, the matrix $\widetilde B_t$ is also of full rank assuming $\widetilde B$ is.

    Suppose that $C$ is not of affine type $A$. In particular the underlying graph $G(\Omega)$ is a tree. Then there is always a sink/source mutation sequence such that $B_t$ is bipartite. Denote the $n$ vertices in $\mathbb T_n$ adjacent to $t$ by $t \frac{i}{\quad\quad} t_i$ for $i = 1, \dots, n$. Using \Cref{prop: reflection generic cc function coef} again, we know that any $X^{\widetilde B_t}_{\tilde {\mathbf g}_t}$ is a Laurent polynomial in any of the seeds $t_i$, that is, 
    \[
        X^{\widetilde B_t}_{\tilde{\mathbf g}_t} \in \mathbb Z[x_{1;t_i}^\pm, \dots, x_{m; t_i}^\pm]\quad \text{ for $i = 1, \dots, n$}.
    \]
    It follows from \cite[Corollary 1.9]{MR2110627} that $X^{\widetilde B_t}_{\tilde{\mathbf g}_t}$ belongs to the \emph{upper cluster algebra}
    \[
        \overline{\mathcal A}(\widetilde B) \coloneqq \bigcap_{v\in \mathbb T_n} \mathbb Z[x_{1;v}^\pm, \dots, x_{m; v}^\pm],
    \]
    which is known to equal $\mathcal A(\widetilde B)$ \cite[Theorem 1.18]{MR2110627}.

    Finally, let $C$ be of affine type $A$. According to \Cref{ex: affine A generic f poly}, the generic CC function $X^{\widetilde{B}}_{\widetilde{\mathbf g}}$ is independent of the symmetrizer $D$. In the case of minimal $D$, the desired property is known; see \cite{MR2738377,MR3036003} for the coefficient-free case and \cite{MR3061943} for the current situation where $\widetilde B$ is of full rank.
\end{proof}

\begin{proof}[{Proof of \Cref{thm: generic basis}}]
    First of all by \Cref{lemma: generic cc function univ laurent}, the set $\mathcal S$ is indeed a subset of $\mathcal A(\widetilde B)$. We then follow the same strategy as in Qin's proof of \cite[Theorem 1.2.3]{MR4721032} that generic CC functions form a basis in the skew-symmetric case (in a much more general context allowing non-acyclic seeds). Consider the source mutation sequence (\ref{eq: source sequence})
    \[
        v_0 = t_0 \frac{n}{\quad\quad} v_1 \frac{n-1}{\quad\quad} \cdots \frac{2}{\quad\quad} v_{n-1} \frac{1}{\quad\quad} v_n.
    \]
    This is called an \emph{injective-reachable} mutation sequence in \cite{MR4721032}. Heuristically, the cluster variables in seed $v_n$ correspond to the injective modules $\inj_1, \dots, \inj_n$ in our construction in \Cref{subsec: cluster variables cc function}. According to \cite[Theorem 4.3.1]{MR4721032} (with the full-rank assumption of $\widetilde B$), to prove that $\mathcal S$ is a $\mathbb Z$-basis of $\overline{\mathcal A}(\widetilde B) = \mathcal A(\widetilde B)$, it suffices to check that every $X_{\tilde {\mathbf g}}^{\widetilde B}$ is \emph{compatibly pointed} \cite[Definition 3.4.2]{MR4721032} at seeds $v_0$, $v_1$, \dots, $v_n$. Precisely, this means
    \begin{enumerate}
        \item for any $v_s$, $X_{\tilde{\mathbf g}}^{\widetilde B}$ is of the form
        \[
            P_{v_s}(\hat y_{1; v_s}, \dots, \hat y_{n; v_s}) \cdot \prod_{i=1}^m x_{i; v_s}^{g_{i; v_s}}
        \]
        where $P_{v_s}$ is a polynomial of $n$ variables with constant term 1 (\emph{Pointed});
        \item for $v_s \frac{n-s}{\quad\quad}v_{s+1}$, the vectors $\tilde {\mathbf g}_{v_s} \coloneqq (g_{i; v_s})_{i=1}^m$ and $\tilde {\mathbf g}_{v_{s+1}} \coloneqq (g_{i; v_{s+1}})_{i=1}^m$ are related by the mutation rule (\ref{eq: mutation rule extended g vector}).
    \end{enumerate}
    Starting from $\tilde {\mathbf g}_{v_0} \coloneqq \tilde{\mathbf g}$, we iteratively apply the operation $T_k$ in (\ref{eq: mutation rule extended g vector}) to define $\tilde{\mathbf g}_{v_s}$ such that
    \[
        \tilde {\mathbf g}_{v_{s+1}} = T_{n-s}(\tilde {\mathbf g}_{v_s}).
    \]
    It then follows inductively from \Cref{prop: reflection generic cc function coef} that at each seed $v_s$, $X_{\widetilde{\mathbf g}}^{\widetilde B}$ is pointed at $\tilde{\mathbf g}_{v_s}$ with $P_{v_s}$ the corresponding generic $F$-polynomial. Now that we have checked the \emph{compatibly pointed} property of every element in $\mathcal S$, this finishes the proof that $\mathcal S$ is a basis.

    Lastly we show that $\mathcal S$ contains all cluster monomials. We have understood that if $\mathbf g = (g_i)_{i\in I}\in \mathbb Z^n$ is the injective $\mathbf g$-vector of a (decorated) rigid locally free $H$-module $\mathcal M = (M, \mathbf v)$, then 
    \[
        F_M\left(y_i\prod_{j=1}^n x_j^{b_{ji}}\middle\vert i\in I\right) \cdot \prod_{i = 1}^n x_i^{g_i}    
    \]
    is the corresponding cluster monomial in the cluster algebra $\mathcal A_\bullet(B)$ with principal coefficients; see \Cref{rmk: cc function of cluster monomial}. Then by Fomin--Zelevinsky's \emph{separation formula} (\Cref{thm: separation formula}), there exists some $\tilde {\mathbf g}\in \mathbb Z^m$ extending $\mathbf g$ such that $X^{\widetilde B}_{\tilde {\mathbf g}}$ is the corresponding cluster monomial; see \Cref{rmk: extended g vector trop eva}. This finishes the proof.
\end{proof}

\subsection{Canonical decomposition}\label{subsec: canonical decomposition}

We write $\mathcal Z(\mathbf r) = \replf(H, \mathbf r)$. Suppose that $\mathbf r = \mathbf r_1 + \dots + \mathbf r_t$. Let $\mathcal Z(\mathbf r_1) \oplus \cdots \oplus \mathcal Z(\mathbf r_t)$ denote the set of points in $\mathcal Z(\mathbf r)$ that are isomorphic to $M_1\oplus \cdots \oplus M_t$ for some $M_i\in \mathcal Z(\mathbf r_i)$ for $1\leq i \leq t$.

\begin{theorem}[\cite{MR3801499}]\label{thm: can decomp H}
    Let $H = H_K(C, D, \Omega)$ with $K$ an algebraically closed field. For any $\mathbf r\in \mathbb N^n\setminus \{0\}$, there exist $\mathbf r_1, \dots, \mathbf r_t \in \mathbb N^n\setminus \{0\}$ such that
    \begin{enumerate}
        \item $\mathbf r = \mathbf r_1 + \cdots + \mathbf r_t$;
        \item $\mathcal Z(\mathbf r) = \overline{\mathcal Z(\mathbf r_1) \oplus \cdots \oplus \mathcal Z(\mathbf r_t)}$;
        \item each $\mathcal Z(\mathbf r_i)$ contains generically indecomposable (locally free) modules.
    \end{enumerate}
    Moreover, the tuple $(\mathbf r_i)_i$ is unique and independent of the symmetrizer $D$ up to permutations.
\end{theorem}

The decomposition $\mathbf r = \mathbf r_1 + \cdots + \mathbf r_t$ is called the \emph{canonical decomposition} of $\mathbf r$, studied initially by Kac \cite{MR677715} in the context of quiver representations. It follows directly from the above theorem that a generic locally free module $M$ is a direct sum $M_1 \oplus \cdots \oplus M_t$ where each $M_i$ is also generic in $\mathcal Z(\mathbf r_i)$. Consequently we have

\begin{proposition}\label{prop: generic f poly can decomp}
Let $\mathbf r = \mathbf r_1 + \cdots + \mathbf r_t$ be the canonical decomposition of $\mathbf r$. Then the generic $F$-polynomial $F_\mathbf r$ on $\mathcal Z(\mathbf r)$ equals $\prod_{i=1}^t F_{\mathbf r_i}$ where each $F_{\mathbf r_i}$ is the generic $F$-polynomial on $\mathcal Z(\mathbf r_i)$.   
\end{proposition}

\begin{proof}
    Let $U_i\in \mathcal Z(\mathbf r_i)$ where the generic $F$-polynomial $F_{\mathbf r_i}$ is taken. Consider the map
    \[
        \mathrm{GL}_\mathbf r(H) \times \mathcal Z(\mathbf r_1)\times \cdots \times \mathcal Z(\mathbf r_t) \rightarrow \mathcal Z(\mathbf r)
    \]
    sending a tuple $(g, M_1, \dots, M_t)$ to $g\cdot(M_1\oplus \cdots \oplus M_t)$. Then \Cref{thm: can decomp H} says that the image of the map is dense. It follows that the image of the open subset $\mathrm{GL}_\mathbf r(H) \times U_1 \times \cdots \times U_t$ is also dense. Therefore we can find $M \cong M_1\oplus \dots \oplus M_t$ such that $F_M = F_\mathbf r$ with $M_i\in U_i$. By \Cref{rmk: cc function direct sum}, we have $F_\mathbf r = \prod_{i=1}^tF_{\mathbf r_i}$.
\end{proof}

When $C$ is of affine type, the canonical decomposition for $H$ has recently been worked out by Pfeifer.


\begin{theorem}[\cite{pfeifer2023generic}]\label{thm: pfeifer}
    The canonical decomposition of a rank vector $\mathbf r\in \mathbb N^n$ is of the form 
    \[
        \mathbf r = \eta_1 + \dots + \eta_m + \sum_{i = 1}^t \beta_i
    \]
    where $\eta_1 = \dots = \eta_m = \eta$ is the minimal imaginary root, and $\beta_i$ is the rank of an indecomposable locally free rigid module $M(\beta_i)$ such that
    \begin{enumerate}
        \item when $m = 0$, $M(\beta_1)\oplus \cdots \oplus M(\beta_t)$ is rigid;
        \item when $m\neq 0$, $M(\beta_1)\oplus \cdots \oplus M(\beta_t)$ is regular and rigid.
    \end{enumerate}
\end{theorem}

\begin{remark}
    Let $\tilde {\mathbf g}\in \mathbb Z^m$ and $\mathbf v\in \mathbb Z^n$ correspond to the principal part $\mathbf g$ of $\tilde {\mathbf g}$. Then $\mathbf v = \mathbf v^+ - \mathbf v^-$ and we can canonically decompose $\mathbf v^+$ as in \Cref{thm: pfeifer}. It then follows that the generic CC function $X_{\tilde {\mathbf g}}$ has a factorization into a cluster monomial and a monomial of $X_{\xi}$ (up to a monomial of frozen variables) where $\xi\in \mathbb Z^m$ is some extended $\mathbf g$-vector whose principal part is $-\frac{1}{2}B\eta$, the $\mathbf g$-vector of $\eta$.
\end{remark}

According to \Cref{thm: pfeifer} and \Cref{prop: generic f poly can decomp}, besides $F_{\beta_i}$ (which is the $F$-polynomial of the corresponding cluster variable), it only amounts to knowing $F_\eta$ to express $F_\mathbf r$ (thus also $X_\mathbf r$). In fact, when $D$ is minimal, Pfeifer constructs in \cite{pfeifer2023generic} a $\mathbb P^1$-family of non-isomorphic modules $V_\lambda$, $\lambda\in \mathbb P^1$ such that the union of their $\mathrm{GL}_\mathbf r(H)$-orbits $\bigsqcup_{\lambda\in \mathbb P^1} \mathcal O(V_\lambda)$ is dense in $\replf(H, \eta)$. Therefore the generic $F$-polynomial can be taken in $\{V_\lambda \mid \lambda\in \mathbb P^1\}$. The corresponding generic CC function has been computed in \cite[Remark 4.8]{pfeifer2023generic} using an explicit description of $V_\lambda$ in the case where $C = \begin{bsmallmatrix}
    2 & -1 \\
    -4 & 2
\end{bsmallmatrix}$.

\begin{example}
    We continue with \Cref{ex: affine b3}. The minimal imaginary root is $\eta = (1, 1, 1, 1)$. A generic module of rank vector $\eta$ is of the form
    \[
        V_\lambda = \begin{tikzcd}
            4 \ar[r] & 3 \ar[d] \ar[r] & 2 \ar[d] \ar[r] & 1\\
            & 3 \ar[r] & 2 \ar[ru, "\lambda"] & 
        \end{tikzcd}.
    \]
    The $F$-polynomial is clearly $1 + y_1 + y_1y_2 + y_1y_2y_3 + y_1y_2y_3y_4$ and the (injective) $\mathbf g$-vector is $(-1, 0, 0, 1)$.
\end{example}

\begin{example}\label{ex: affine A generic f poly}
    Let $C$ be of affine type $A$ and let $\Omega$ be an acyclic orientation of $C$. By \Cref{thm: can decomp H}, the canonical decomposition of any rank vector $\mathbf{r}$ is independent of the symmetrizer $D$; in particular we may take the minimal symmetrizer $D = \mathrm{diag}(1, \dots, 1)$. It is well-known that, in this case, each rank vector $\mathbf r_i$ appearing in the canonical decomposition of $\mathbf r$ as components only in $\{0, 1\}$ \cite{MR0447344,MR677715}. It follows immediately from \Cref{prop: generic f poly can decomp} that the generic $F$-polynomial $F_\mathbf{r}$ is independent of the choice of $D$.
\end{example}

\section*{Acknowledgements}

LM was supported by the Deutsche Forschungsgemeinschaft (DFG, German Research Foundation) – Project-ID 281071066 – TRR 191. LM and XS received supports from the Royal Society by the Newton International Fellowships Alumni AL$\backslash$231002.

\providecommand{\bysame}{\leavevmode\hbox to3em{\hrulefill}\thinspace}
\providecommand{\MR}{\relax\ifhmode\unskip\space\fi MR }
\providecommand{\MRhref}[2]{%
  \href{http://www.ams.org/mathscinet-getitem?mr=#1}{#2}
}
\providecommand{\href}[2]{#2}

\end{document}